\documentclass[11pt,twoside,a4paper,reqno]{amsart}
\usepackage[english]{babel}
\usepackage{a4wide}
\usepackage{graphicx}
\usepackage{amsmath,amssymb,amsthm,amsgen,mathtools}
\usepackage{stmaryrd}
\usepackage{listings}
\usepackage{color}
\usepackage[usenames,dvipsnames]{xcolor}
\usepackage{rotating}
\usepackage{hhline}
\usepackage{multirow}
\usepackage{enumerate}
\usepackage[bookmarks]{}
\usepackage{amsfonts}   
\usepackage{dsfont}
\usepackage{tikz-cd}
\usepackage{esint}
\usepackage{booktabs}
\usepackage{longtable}
\usepackage{stmaryrd}
\usepackage{hyperref}

\newtheorem{thm}{Theorem}[section]

\newtheorem{problem}[thm]{Problem}
\newtheorem{prop}[thm]{Proposition}
\newtheorem{ass}[thm]{Assumption}
\newtheorem{lem}[thm]{Lemma}

\newtheorem{defn}[thm]{Definition}

\newtheorem{rem}[thm]{Remark}

\newcommand{\AC}{\mathrm{AC}}

\newcommand{\bgamma}{\boldsymbol \gamma}
\newcommand{\bGamma}{\boldsymbol \Gamma}
\newcommand{\bmu}{\boldsymbol \mu}
\newcommand{\bnu}{\boldsymbol \nu}

\newcommand{\brho}{\boldsymbol \rho}

\newcommand{\bW}{\mathbf W}

\newcommand{\rC}{\mathrm{C}}
\newcommand{\col}{\mathrm{col}}

\newcommand{\de}{\mathrm{d}}
\newcommand{\sfd}{\mathsf{d}}

\renewcommand{\div}{\operatorname{div}}

\newcommand{\rL}{\mathrm{L}}
\newcommand{\Lip}{\operatorname{Lip}}
\newcommand{\loc}{\operatorname{loc}}

\newcommand{\N}{\mathbb N}

\newcommand{\R}{\mathbb R}

\newcommand{\supp}{\operatorname{supp}}

\newcommand{\weakto}{\rightharpoonup}
\newcommand{\xto}[1]{\xrightarrow{ #1 }}

\begin{document}

\title[Many-particle limits including annihilation]{Discrete-to-continuum limits of particles with an annihilation rule} 
\author{Patrick van Meurs}
\address{Institute of Liberal Arts and Science,
Kanazawa University,
Kakuma, Kanazawa, 920-1192, 
Japan}
\email[P.~van Meurs]{pjpvmeurs@staff.kanazawa-u.ac.jp}
\author{Marco Morandotti}
\address{Dipartimento di Scienze Matematiche ``G.~L.~Lagrange'', Politecnico di Torino, Corso Duca degli Abruzzi, 24, 10129 Torino, Italy}
\email[M.~Morandotti]{marco.morandotti@polito.it}
\date{\today}

\begin{abstract}
In the recent trend of extending discrete-to-continuum limit passages for gradient flows of single-species particle systems with singular and nonlocal interactions to particles of opposite sign, any annihilation effect of particles with opposite sign has been side-stepped. We present the first rigorous discrete-to-continuum limit passage which includes annihilation. This result paves the way to applications such as vortices, charged particles, and dislocations.
In more detail, the discrete setting of our discrete-to-continuum limit passage is given by particles on the real line. 
Particles of the same type interact by a singular interaction kernel; those of opposite sign interact by a regular one. 
If two particles of opposite sign collide, they annihilate, \emph{i.e.}, they are taken out of the system. 
The challenge for proving a discrete-to-continuum limit is that annihilation is an intrinsically discrete effect where particles vanish instantaneously in time, while on the continuum scale the mass of the particle density decays continuously in time.
The proof contains two novelties: (i) the observation that empirical measures of the discrete dynamics (with annihilation rule) satisfy the continuum evolution equation that only implicitly encodes annihilation, and (ii) the fact that, by imposing a relatively mild separation assumption on the initial data, we can identify the limiting particle density as a solution to the same continuum evolution equation. 
\end{abstract}

\maketitle




\section{Introduction}\label{sect:intro}

A recent trend in discrete-to-continuum limit passages in overdamped particle systems with singular and nonlocal interactions (with applications to, \emph{e.g.}, vortices 
\cite{Schochet96,
Hauray09,
Duerinckx16}, 
charged particles
\cite{SandierSerfaty152D}, 
dislocations
\cite{HallChapmanOckendon10,
LMSZ18,
MoraPeletierScardia17}, 
and dislocation walls
\cite{GeersPeerlingsPeletierScardia13,
VanMeursMunteanPeletier14,
VanMeursMuntean14}) 
is to extend such results to two-species particle systems. The singularity in the interaction potential imposes the immediate problem that the evolution of the particle system is only defined up to the first collision time between particles of opposite sign. This problem is dealt with by either \emph{regularising} the singular interaction potential (see 
\cite{GarroniLeoniPonsiglione10,
GarroniVanMeursPeletierScardia18prep}) 
or by limiting the geometry such that particles of opposite sign cannot collide (see 
\cite{ChapmanXiangZhu15,
vanMeurs18}). 
However, more realistic models of vortices, charged particles, and dislocations include the \emph{annihilation} of particles of opposite sign. While annihilation has been analysed on the discrete scale
\cite{SmetsBethuelOrlandi07,
Serfaty07II} 
and continuum scale 
\cite{BilerKarchMonneau10,
AmbrosioMaininiSerfaty11} 
separately, there is no rigorous discrete-to-continuum limit passage known between these two scales.

The main result in this paper establishes the first result on a discrete-to-continuum limit passage in two-species particle systems in one dimension with annihilation.

After introducing the discrete 
and continuum problems, we present our main result on the connection between them, \emph{i.e.}, the limit as the number of particles $n$ tends to $\infty$. Then, we put our discrete and continuum problems in the perspective of the literature, and comment how our proof combines known techniques with novel ideas. 
We conclude with an exposition of possible extensions to work towards singular interspecies interactions and higher dimensions.

\subsection{The discrete problem (particle system with annihilation)}
\label{s:in:Pn}

We introduce our discrete evolution problem by first specifying the state of the system, then the related interaction energy, and finally the evolution law. 
The state of the system is described by $x \coloneqq (x_1,\ldots,x_n) \in \R^n$ and $b \coloneqq (b_1, \ldots, b_n) \in \{-1, 0, 1\}^n$ with $n \geq 2$ the number of particles. 
The point $x_i$ is the location of the $i$-th particle, and $b_i$ is its charge {(or Burgers vector, in the setting of dislocations)}.

To any state $(x, b)$ we assign the interaction energy $E_n\colon \R^n\times\{-1,0,1\}^n\to\R\cup\{+\infty\}$ by
\begin{equation}\label{101}
  E_n(x; b) 
  \coloneqq \frac1{2 n^2} \sum_{i = 1}^n \bigg( \sum_{\substack{ j = 1 \\ j \neq i \\ b_i b_j = 1 }}^n V (x_i - x_j) 
     + \sum_{\substack{ j = 1 \\ b_i b_j = -1 }}^n W(x_i - x_j) \bigg),
\end{equation}
where $V$ and $W$ are the interaction potentials between particles of equal and opposite charge, respectively. For $V$ and $W$, we have three choices in mind, all of which are of separate interest:
\begin{enumerate}[(i)]
  \item $V(r) = -\log |r|$ and $W\equiv 0$. This corresponds to the easiest case in which the two species only interact with their own kind. It is distinct from the single-particle case solely by the annihilation rule which we specify below.
  \item $V(r) = -\log |r|$ and $W$ a regularisation of $-V$ (as illustrated in Figure~\ref{fig:VW}). This is a first step to considering the case of positive and negative charges (or positive and negative dislocations). After stating our main result for regular $W$, we comment in Section~\ref{s:in:conc} on possible extensions to singular $W$, in particular $W = -V$.
  \item $V(r) = r \coth r - \log |2 \sinh r|$ and $W$ a regularisation of $-V$. This setting corresponds to that of dislocation walls, whose discrete-to-continuum limit is established in 
  \cite{HirthLothe82,
  Hall11,
  GeersPeerlingsPeletierScardia13,
  VanMeursMunteanPeletier14,
  VanMeursMuntean14,
  vanMeurs18} for either single-sign scenarios or without annihilation. The potential $V$ has several pleasant properties: it has a logarithmic singularity at $0$, it is decreasing on $(0,\infty)$, and it is positive with integrable tails.
\end{enumerate}

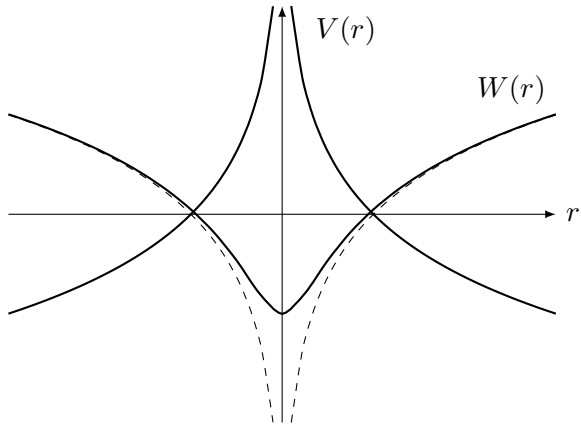
\begin{figure}[h]
\centering
\begin{tikzpicture}[scale=1.2, >= latex]    
\def \w {3}
\def \y {2.3} 
\def \d {1/3}

\draw[->] (-\w,0) -- (\w,0) node[right] {$r$};
\draw[->] (0,-\y) -- (0,\y);
\draw[domain=0.1:\w, smooth, thick] plot (\x,{ -ln(\x) });
\draw[domain=-\w:-.1, smooth, thick] plot (\x,{ -ln(-\x) });
\draw[domain=0.1:\w, smooth, dashed] plot (\x,{ ln(\x) });
\draw[domain=-\w:-.1, smooth, dashed] plot (\x,{ ln(-\x) });
\draw[domain=-\w:\w, smooth, thick] plot (\x,{ ln(\x*\x + \d*\d)/2 });
\draw (.7,\y) node[below] {$V(r)$};
\draw (\w,1.1) node[anchor = south east] {$W(r)$};
\end{tikzpicture} \\
\caption{Plots of $V(r) = -\log|r|$ and a typical regularisation $W$ of $-V$.}
\label{fig:VW}
\end{figure}

We propose a unified setting which includes the three cases above: 
we consider a class of potentials $V$ and $W$ which satisfy a certain set of assumptions specified in 
Assumption~\ref{ass:VW}. The crucial assumptions are that the singularity of $V$ at $0$ is at most logarithmic, that $V(r) \to +\infty$ as $r \to 0$, that $W$ is regular, and that $V$ and $W$ have at most logarithmic growth at infinity. In view of other typical assumptions in the literature, we \emph{do not} rely on convexity or monotonicity. In Section~\ref{s:in:disc} we elaborate on the necessity of these assumptions to our main discrete-to-continuum result. 

Finally, we make three observations on the structure of \eqref{101}. 
First, if the $i$-th particle has 0 charge (\emph{i.e.}, $b_i = 0$), then it does not contribute to $E_n$. Second, the factor $1/2$ in front of the energy is common; it corrects the fact that all interactions are counted twice in the summation. Third, the condition $j \neq i$ prevents self-interaction.
\smallskip

Equation \eqref{Pn:intro} formally describes the dynamics; for a rigorous definition see Problem~\ref{problemPn} and Definition~\ref{defSolPn}. 
   \begin{equation} \label{Pn:intro} 
   \left\{ \begin{aligned}
     &\frac \de{\de t} x_i = - \frac1{n} \sum_{ j \,:\, b_i b_j = 1 } V' (x_i - x_j) - \frac1{n} \sum_{ j\,:\, b_i b_j = -1 } W' (x_i - x_j)
     &&\text{on $(0, T) \setminus T_{\col}$,} \\
     &\text{annihilation rule at $T_{\col}$.}
     &&
   \end{aligned} \right.
\end{equation}
Here, $T_{\col} = \{t_1, \ldots, t_K\}$ is a finite set, outside of which $x(t)$ is the gradient flow of $E_n$. The collision times $t_k$ correspond to the times at which at least one pair $(i,j)$ of two particles collide, \emph{i.e.} $x_i(t_k) = x_j (t_k)$. We will show that the singularity of $V$ keeps particle of the same type separated, which implies that the only possible particle collisions are that of two particles with opposite charge. The annihilation rule dictates that at a collision between particles $x_i$ and $x_j$, the charges $b_i(t)$ and $b_j(t)$ are put to $0$ for all $t \geq t_k$. After $t_k$, the system of ODEs is restarted with initial condition $(x(t_k); b(t_k))$. While particles are not removed from the dynamics, we note that, if particle $i$ has zero charge, then
\begin{itemize}
  \item $x_i (t)$ remains stationary,
  \item the velocity of all other particles does not depend on $x_i(t)$, and
  \item particle $i$ cannot annihilate any more with any other particle.
\end{itemize}
Hence, the mathematical framework of \eqref{Pn:intro} encodes 
annihilations without removing particles from the equations.

\subsection{The continuum problem (PDE for the particle density)}
\label{s:in:P}

On the continuum level, the state of the system is described by the nonnegative measures $\rho^\pm$, which represent the density of the positive/negative particles (including those that are annihilated). We further set 
\begin{equation*}
  \rho := \rho^+ + \rho^-
  \quad \text{and} \quad
  \kappa := \rho^+ - \rho^-,
\end{equation*}
and require the total mass of $\rho$ to be $1$. We note that $\rho^+$ and $\rho^-$ need not be mutually singular, and thus $\rho^\pm \geq [\kappa]_\pm$, where $[\kappa]_\pm$ denotes the positive/negative part of the signed measure $\kappa$. We interpret $[\kappa]_\pm$ as the density of positive/negative particles that have not been annihilated yet. 

For $\rho^\pm(t)$ we consider the following set of evolution equations
\begin{equation} \label{P}
  \left\{ \begin{aligned}
    & \partial_t \rho^+ = \div \big( [\kappa]_+ \, (V' * [\kappa]_+ + W' * [\kappa]_- ) \big)
    && \text{in } \mathcal D' ( (0,T) \times \R ), \\
    & \partial_t \rho^- = \div \big( [\kappa]_- \, (V' * [\kappa]_- + W' * [\kappa]_+ ) \big)
    && \text{in } \mathcal D' ( (0,T) \times \R ),
  \end{aligned} \right.
\end{equation}
We remark that no annihilation rule is specified; the annihilation is encoded in taking the positive/negative part of $\kappa$. Indeed, it is easy to imagine that while $\rho = \rho^+ + \rho^-$ is conserved in time, $[\kappa]_+ + [\kappa]_- = |\rho^+ - \rho^-|$ may not be.

\subsection{Main result: discrete-to-continuum limit}
\label{s:in:t}

Our main theorem (Theorem~\ref{t}) states that the solutions to \eqref{Pn:intro} converge to a solution of \eqref{P} as $n \to \infty$.
It specifies the concept of solution to both problems, the required conditions on the sequence of initial data of \eqref{Pn:intro}, and guarantees that the so-constructed solution to \eqref{P} at time $0$ corresponds to the limit of the initial conditions as $n \to \infty$. The convergence is uniform in time on $[0,T]$ for any $T > 0$. The convergence in space is with respect to the weak convergence.
As a by-product of Theorem~\ref{t}, we obtain global-in-time existence of a solution $(\rho^+, \rho^-)$ to \eqref{P} for which the masses of $\rho^\pm$ are conserved in time.
\smallskip

In order to give effectively an outline of the proof and the motivation for the main assumptions under which Theorem \ref{t} holds (Section \ref{s:in:disc}), we first describe the related literature.

\subsection{Related literature}
\label{s:in:lit}

We start by relating \eqref{P} formally to its singular counterpart. Replacing $W$ by $-V$, we obtain from a formal calculation that the difference of the two equations in \eqref{P} is given by
\begin{equation} \label{P:sing}
  \partial_t \kappa
  = \div \big( |\kappa| ( V'*\kappa ) \big).
\end{equation}
For $V(r) = -\log|r|$, equation \eqref{P:sing} was introduced by \cite{Head72III} and later proven in \cite{BilerKarchMonneau10} to attain unique solutions when posed on $\R$ with proper initial data.
\smallskip

In the remainder of this subsection, we put our main result Theorem~\ref{t} in the perspective of the literature. 
We start by describing those specifications of \cite{ForcadelImbertMonneau09,MonneauPatrizi12,MonneauPatrizi12a} which are closest to our main result. A specification of \cite[Theorems~2.1--2.3]{ForcadelImbertMonneau09} proves a `discrete'-to-continuum result from \eqref{Pn:intro} to \eqref{P:sing}, in the case where $V(r) = - W(r)$ is a regularisation of $-\log|r|$ on the length-scale $1/n$. 
We put `discrete' in apostrophes, because their equivalent of \eqref{Pn:intro}, given by \cite[equation (5)]{ForcadelImbertMonneau09}, is a Hamilton-Jacobi equation, which 
includes the solution to \eqref{Pn:intro} only if all particles have the same sign. 
It is not clear if this Hamilton-Jacobi equation relates to \eqref{Pn:intro}
if the particles have opposite sign. 

As opposed to \cite{ForcadelImbertMonneau09}, \cite{MonneauPatrizi12} starts from a different Hamilton-Jacobi equation, which corresponds to the Peierls-Nabarro model \cite{Nabarro47,Peierls40}. This model is a phase-field model for the dynamics of dislocations which naturally includes annihilation. In this model, opposite to encoding dislocations as points on the line, the dislocations are identified by the pulses of the derivative of a multi-layer phase field on the real line. In \cite{MonneauPatrizi12}, the width of these pulses is taken to be on the same length-scale as the typical distance between neighbouring dislocations. Then, in the joint limit when the regularisation length-scale (and thus simultaneously $1/n$) tend to $0$, an \emph{implicit} Hamilton-Jacobi equation is recovered \cite{MonneauPatrizi12}. In \cite[Theorem~1.2]{MonneauPatrizi12a} it is shown that this implicit Hamilton-Jacobi equation converges to \eqref{P:sing} in the dilute dislocation density limit. While this framework seems promising for a direct `discrete'-to-continuum result (`discrete' being the Peierls-Nabarro model) to \eqref{P}, it only applies to co-dimension $1$ objects, \emph{i.e.}, particles in $1$D and curves in $2$D.
\smallskip

Next we discuss the literature related to problem \eqref{Pn:intro}.
\cite{Serfaty07II} and \cite[Theorems~1.3 and 1.4]{SmetsBethuelOrlandi07} describe the version of \eqref{Pn:intro} in which $W(r)$ is replaced by $-V(r) = \log |r|$. In this setting, a solution is constructed as the limit of the Ginzburg-Landau equation on the dynamics of vortices when the phase-field parameter $\varepsilon$ tends to $0$, and detailed properties of this solution are established. However, we did not find a precise solution concept to this version (or any other version) of \eqref{Pn:intro}, which in particular yields a unique solution. We establish such a solution concept to \eqref{Pn:intro} in Definition \ref{defSolPn} and Proposition~\ref{prop:Pn}.
\smallskip

Regarding the continuum problem \eqref{P}, we have not found this set of equations in the literature. Nonetheless, we believe the case $W=0$ to be of independent interest, since then \eqref{P} serves as the easiest benchmark problem for future studies on annihilating particles. 
Also, since our discrete-to-continuum result holds for taking $W$ as a regularisation of $-V$, we expect that \eqref{P:sing} can be obtained from \eqref{P} as the regularisation length-scale tends to $0$ (see Section~\ref{s:in:conc}). Therefore, we review the literature on \eqref{P:sing}.

Equation \eqref{P:sing} as posed on $\R$ with $V(r) = -\log|r|$, or even $V(r) = |r|^{-a}$ with $0 < a < 1$, attains a self-similar solution \cite[Theorem~2.4]{BilerKarchMonneau10} in which $\kappa$ has a sign. The self-similar solution is expanding in time (due to the repelling interaction force $V'(r)$), and describes the long-time behaviour of the unique viscosity solutions to \eqref{P:sing} \cite[Theorem~2.5]{BilerKarchMonneau10} for appropriate initial data. Moreover, for $V(r) = -\log|r|$ and initial condition $\kappa_\circ \in L^1(\R)$, the viscosity solution $\kappa$ to \eqref{P:sing} satisfies $\kappa(t) \in L^p(\R)$ for all $1 \leq p \leq \infty$ \cite[Theorem~2.7]{BilerKarchMonneau10}. In conclusion, despite \eqref{P:sing} being the singular counterpart of \eqref{P}, it has a well-defined global-in-time solution concept.

Lastly, we compare our result to that of \cite{AmbrosioMaininiSerfaty11}. There, the authors are interested in deriving a gradient flow structure of \eqref{P:sing} on $\R^2$ 
with $V$ having a logarithmic singularity at $0$ by defining a discrete in time minimising movement scheme and passing to the limit as the time step size tends to 0. The related convergence result is \cite[Theorem~1.4]{AmbrosioMaininiSerfaty11}. However, the limit equation is not fully characterised as \eqref{P:sing}, since in that equation $|\kappa|$ is replaced by an unknown measure $\mu \geq |\kappa|$ which is obtained from compactness. 
The connection to our main result is that we faced a similar problem. Due to our 1D setup and by a technical assumption on the initial data, we were able to characterise the corresponding $\mu$ as $|\kappa|$.

\subsection{Discussion on the proof, assumptions, and possible extensions}
\label{s:in:disc}

We divide this section into several topics regarding the proof, assumptions, and possible extensions of Theorem \ref{t} (outlined in Section \ref{s:in:t}). 
\smallskip

\emph{Summary of the proof}. A crucial step is the observation that the solution to \eqref{Pn:intro}, seen as a pair of empirical measures $\mu_n^\pm$, is a solution to \eqref{P}, \emph{i.e.},
\begin{equation} \label{P:mun}
  \left\{ \begin{aligned}
    & \partial_t \mu_n^+ = \div \big( [\kappa_n]_+ \, (V' * [\kappa_n]_+ + W' * [\kappa_n]_- ) \big)
    && \text{in } \mathcal D' ( (0,T) \times \R ), \\
    & \partial_t \mu_n^- = \div \big( [\kappa_n]_- \, (V' * [\kappa_n]_- + W' * [\kappa_n]_+ ) \big)
    && \text{in } \mathcal D' ( (0,T) \times \R ),
  \end{aligned} \right.
\end{equation}
where $\kappa_n := \mu_n^+ - \mu_n^-$. The annihilation is completely covered by taking the positive and negative part of $\kappa_n$. This property is the reason for encoding annihilation in the charges $b_i(t)$ rather than removing particles from the dynamics. Then, relying on the gradient flow structure underlying \eqref{Pn:intro} and the boundedness of $W$, we find, by the usual compactness arguments \emph{\`a la} Arzel\`a-Ascoli, limiting curves $\rho^\pm(t)$. It then remains to pass to the limit $n \to \infty$ in \eqref{P:mun}.
The difficulty is in characterising the limit of $[\kappa_n]_\pm$, which only accounts for the particles that have not collided yet. Indeed, the convergence of measures is not invariant with respect to taking the positive and negative part. It is here that we heavily rely on the one-dimensional setting and on a technical assumption on the initial data (Assumption~\ref{ass:IC}), which provides an $n$-independent bound on the number of neighbouring pairs of particles with opposite sign. This bound allows us to characterise the limit of $[\kappa_n]_\pm$ as $[\kappa]_\pm$. 
\smallskip

\emph{Motivation for Assumption~\ref{ass:IC}}. Assumption~\ref{ass:IC} prevents small-scale oscillations between $\pm 1$ phases. A similar assumption is made in \cite{MonneauPatrizi12}, where the initial data for the particles is constructed from the continuum initial datum. While one might expect that small-scale oscillations cancel out on small time scales, the simulations in \cite[Chapter~9]{PatricksThesis} suggest otherwise. The problem with such small-scale oscillations is that they cause the limit of $[\kappa_n]_\pm$ to be larger than $[\kappa]_\pm$, which makes it difficult to characterise the limit as $n \to \infty$ of \eqref{P:mun} as \eqref{P}.
\smallskip

\emph{Singularity of $V$}. Assuming the singularity of $V$ to be at most logarithmic is needed to apply the discrete-to-continuum limit passage technique in \cite{Schochet96}. 

In fact, we also \emph{require} that $V(r) \to \infty$ as $r \to 0$, \emph{i.e.}, we do not allow for a regular $V$. While regular $V$ and $W$ (in particular $W = -V$) would simplify the equations and many steps in the proof of our main theorem, it may result in two technical difficulties: collision between three or more particles, and the limiting signed measure $\kappa$ having atoms. These difficulties complicate the convergence proof of $[\kappa_n]_\pm$ to $[\kappa]_\pm$ as $n \to \infty$. Since all our intended applications correspond to singular potentials $V$, we choose to side-step these technical difficulties by simply requiring $V$ to have a singularity at $0$.
\smallskip

\emph{Regularity of $W$}. $W$ being bounded around $0$ results in a lower bound on the energy along the evolution, which we need for equicontinuity and thus for compactness of $\mu_n^\pm$. Also, while passing to the limit $n \to \infty$ in \eqref{P:mun}, we need $W'$ regular enough (the technique in \cite{Schochet96} does not apply for logarithmic $W$).
\smallskip

\emph{Logarithmic tails of $V, W$}. While it would be easier to assume that $V$ is bounded from below and $W$ is globally bounded, we \emph{also} allow for logarithmic tails to include all three scenarios in Section~\ref{s:in:Pn}. The logarithmic tails of $V$ and $W$ result in the energy $E_n$ to be unbounded from below. However, following the idea in \cite{Schochet96} to prove a priori bounds on the \emph{moments} of $\mu_n^\pm (t)$, we easily obtain that $E ( \mu_n^\pm (t) )$ is bounded from below by $-C(1+t)$ for some $C > 0$ independent of $n$ and $t$.

\subsection{Conclusion and outlook}
\label{s:in:conc}

We intend our main result to open a new thread of research on including annihilation in discrete-to-continuum limits. Here we discuss several open ends.
\smallskip

\emph{$W = -V$ singular}. This setting corresponds to charges (or dislocations) on the real line. On the continuum level, see \eqref{P:sing}, this equation is well-understood \cite{BilerKarchMonneau10}, but on the discrete level we have not found a closed set of equations to describe the discrete counterpart of \eqref{Pn:intro} (other than \cite{Serfaty07II}, \cite{SmetsBethuelOrlandi07}, whose results are discussed in Section~\ref{s:in:lit}). Since our main result does allow for $-W$ to be a regularisation $V_\delta$ of $V$ ($\delta$ denotes the arbitrarily small, but fixed, length-scale of the regularisation), this calls for three interesting limit passages:
\begin{enumerate}[(a)]
  \item \emph{$\delta \to 0$ with $n$ fixed}. This limit seems the easiest out of the three. Similar to \cite{Serfaty07II}, \cite{SmetsBethuelOrlandi07}, the idea is to pass to the limit, and \emph{describe} the limit rather than posing a closed set of equations for it. One challenge is that in the limiting curves prior to collision at $t_*$, the particles' speed blows up as $\sim 1/\sqrt{t_* - t}$ (this is easily seen by considering only 2 particles; one positive and one negative). While the resulting curves are not Lipschitz in time, they are $\mathrm{C}^{1/2}$ in time. However, such collisions correspond to $-\infty$ wells in the energy, which require the development of a proper renormalisation of $E_n$.

Another challenge is that particles need not collide if they come close, regardless how small $\delta > 0$ is. To see this, consider two particles with opposite sign and with mutual distance smaller than $\delta$. Since $V_\delta$ is regular, the particles will come exponentially close, but they will not collide in finite time. In the case of many particles, such a close pair will only collide if the external force (induced by the other particles) acts in the right direction. If it does not collide, then the pair remains in the system (as opposed to the case of singular $W$), and may even interact with or annihilate other particles that come close.
  
  \item \label{bullet:b} \emph{Connecting \eqref{P} to \eqref{P:sing} by $\delta \to 0$}. Taking $W = -V_\delta$ and setting $\rho_\delta^\pm$ as a corresponding solution to \eqref{P}, it is impossible to pass directly to the limit in \eqref{P} due to the term $[\kappa_\delta]_\pm (V_\delta' * [\kappa_\delta]_\mp)$. Instead, the structure of \eqref{P:sing} in terms of viscosity solutions (see \cite{BilerKarchMonneau10}) seems promising. We leave it to future research to find out whether \eqref{P} enjoys a similar structure, and if not, whether there is a different continuum model for annihilating particles that does.
  
  \item \emph{Connecting \eqref{Pn:intro} to \eqref{P:sing} by a joint limit $n \to \infty$ and $\delta_n \to 0$}. This approach fits to the convergence result obtained in \cite{MonneauPatrizi12}, where roughly speaking $\delta_n \sim 1/n$ is considered, but where a different equation than \eqref{P:sing} is obtained in the limit. It would be interesting to see whether those results can be extended to the case $\delta_n \ll 1/n$, in which case the expected limit is \eqref{P:sing} (see \cite{MonneauPatrizi12a}).
\end{enumerate}
\smallskip

\emph{Different regularisations of collisions}. In the spirit of proving any of the above limit passages, we discuss alternative regularisations other than taking $W$ regular. One idea is `premature annihilation', where particles are removed from the system when they come $\delta$-close, with $\delta > 0$ a regularisation parameter. This approach is commonly adapted in numerical simulations of discrete systems with an annihilation rule. However, it is not obvious what the limiting equation as $n \to \infty$ (counterpart of \eqref{P:sing}) is for $\delta > 0$ fixed, because we expect the supports of $[\kappa]_+$ and $[\kappa]_-$ to be separated by at least $\delta$. A third option is to mollify the jump of the charge $b_i(t)$ from $\pm1$ to $0$, possibly by an additional ODE for $b_i(t)$. We have not found a proper rule for this that would still allow for a discrete-to-continuum convergence result.
\smallskip

\emph{Higher dimensions}. In this paragraph we consider the extension to two dimensions; the discussion easily extends to higher dimensions. The one ingredient in our proof which intrinsically relies on our 1D setting, is the \emph{separation} condition on the initial data. This condition limits the collisions to happen only at a finite number of points. In 2D, collisions are bound to happen along curves (or more complicated subsets of $\R^2$), which makes it challenging to characterise the limit of $[\kappa_n]_\pm$. A similar problem occurred in \cite{AmbrosioMaininiSerfaty11} as discussed in Section~\ref{s:in:lit}. In future research we plan to relax our `separation' assumption, possibly by considering a different regularisation of collisions. 

The remainder of the paper is organised as follows. In Section~\ref{s:not} we fix our notation and list the assumptions on $V$, $W$ and the initial data. In Section~\ref{sect:preliminary} we recall known results and provide the preliminaries. In Section~\ref{s:Pn} we give a rigorous definition of \eqref{Pn:intro}, show that it attains a unique solution, and establish several properties of it. In Section~\ref{s:t} we state and prove our main result, Theorem~\ref{t}.

\section{Notation and standing assumptions}
\label{s:not}


Here we list the symbols and notation which we use in the remainder of this paper: 
\newcommand{\specialcell}[2][c]{\begin{tabular}[#1]{@{}l@{}}#2\end{tabular}}
\begin{longtable}{lll}
$\mathcal B (\R)$ & space of Borel sets on $\R$ & Section~\ref{sect:preliminary}\\
$f(a-)$ & $\lim_{y \uparrow a} f(y)$ &\\
$[f]_\pm$ & positive or negative part of $f$ & \\ 
$\mu \otimes \nu$ & product measure; $(\mu \otimes \nu) (A \times B) = \mu(A) \nu (B)$ & Section~\ref{sect:preliminary} \\
$C>0$ & constant whose value can possibly change from line to line & \\
$\bmu$ & $\bmu \coloneqq (\mu^+, \mu^-) \in \mathcal P (\R \times \{\pm1\})$ & \eqref{bmu:def} \\
$\mathcal M (\R)$ & space of finite, signed Borel measures on $\R$ & Section~\ref{sect:preliminary} \\
$\mathcal M_+ (\R)$ & space of the non-negative measures in $\mathcal M (\R)$ & Section~\ref{sect:preliminary} \\
$\N$ & $\{1,2,3,\ldots\}$ &  \\
$\mathcal P (\R)$ & \specialcell[t]{space of probability measures; \\$\mathcal P (\R) = \{ \mu \in \mathcal M_+ (\R) : \mu (\R) = 1 \}$} & Section~\ref{sect:preliminary} \\
$\mathcal P_2 (\R)$ & \specialcell[t]{probability measures with finite second moment; \\$\mathcal P_2 (\R) = \{ \mu \in \mathcal P_2 (\R) : \int_{-\infty}^\infty x^2 \, \de \mu(x) < \infty \}$} & Section~\ref{sect:preliminary} \\
$V$ & interaction potential for equally signed particles & Assumption~\ref{ass:VW} \\
$W$ & interaction potential for oppositely signed particles & Assumption~\ref{ass:VW} \\
$W(\mu, \nu)$ & $2$-Wasserstein distance between $\mu, \nu \in \mathcal P (\R)$ & \cite{AmbrosioGigliSavare08} \\
$\bW(\bmu, \bnu)$ & 2-Wasserstein distance between $\bmu, \bnu \in \mathcal{P}_2(\R)$ & \eqref{bW2}  
\end{longtable} 

Assumption~\ref{ass:VW} lists the standing properties which we impose on $V$ and $W$.

\begin{ass} \label{ass:VW}
We require that the interaction potentials $V\colon\R\setminus\{0\}\to\R$ and $W\colon\R\to\R$ satisfy the following conditions:
\begin{subequations}\label{ass:VW:eq}
\begin{eqnarray}
 && \text{$V \in \rC^1 (\R \setminus \{0\})$, $W \in \rC^1 (\R)$, $V' \in \Lip_{\loc} (\R \setminus \{0\})$, and $W' \in \Lip (\R)$,} \label{ass:VW:regy} \\
&& \text{$V$ and $W$ are even;} \label{ass:VW:even}  \\
&& \text{$V(r) \to +\infty$ as $r \to 0$;} \label{ass:VW:sing} \\
&& \text{$r \mapsto r V' (r)$ and $r \mapsto r W'(r)$ are in $L^\infty (\R)$.} \label{ass:VW:pbds} 
\end{eqnarray}
\end{subequations}
\end{ass}

For convenience, we set $V'(0) \coloneqq 0$. Below we list two remarks on Assumption~\ref{ass:VW}:
\begin{itemize}
  \item we assume no monotonicity on $V$ or $W$;
  \item Condition \eqref{ass:VW:pbds} implies that $V$ has at most a logarithmic singularity, and that $V$ and $W$ have at most logarithmically diverging tails, namely
  \begin{equation}\label{402}
  |V(r)| + |W(r)| \leq C\big( \big| \log |r| \big| + 1 \big),
  \quad \text{for all } r \neq 0.
  \end{equation}
  Due to condition \eqref{ass:VW:sing}, we can sharpen this inequality around $0$ by
  \begin{equation}\label{VWgrowth}
(V+W)(r)\geq -C r^2,
\quad \text{for all } r \neq 0;
\end{equation}
\end{itemize}

The following assumption on the initial data states that no pair of particles of opposite sign should start at the same position. 
\begin{ass}[Separation assumption on the initial data $(x^\circ; b^\circ)$] \label{ass:IC}
There exist $-\infty < a_0 \leq a_1 \leq \ldots \leq a_{2L} < +\infty$ such that
\begin{equation*}
  \{ x_i^\circ : b_i^\circ = 1 \} 
  \subset \bigcup_{\ell = 1}^L (a_{2 \ell - 2}, a_{2 \ell - 1}), \qquad
  \{ x_i^\circ : b_i^\circ = -1 \} 
  \subset \bigcup_{\ell = 1}^L (a_{2 \ell - 1}, a_{2 \ell}).
\end{equation*}
\end{ass}
The importance of this assumption is clarified later when the limit 
$n \to \infty$ is considered, in which the 
number $L$ is assumed to be $n$-independent (see also Section \ref{s:in:disc}). Moreover, we will show in Proposition~\ref{prop:Pn} that this assumption is conserved in time.

\section{Preliminary results}\label{sect:preliminary}
We collect here some basic definitions and known results that will be useful in the sequel.

\subsection{Probability spaces and the Wasserstein distance}

On $\mathcal P_2 (\R)$ (space of probability measures with finite second moment; see Section~\ref{s:not}), the (square of the) 2-Wasserstein distance between $\mu, \nu \in \mathcal P_2 (\R)$ is defined as 
\begin{align} \label{W2}
W^2 ( \mu, \nu ) \coloneqq \inf_{\gamma \in \Gamma(\mu,\nu)} \iint_{\R^2} |x-y|^2 \, \de \gamma(x,y),
\end{align}
where $\Gamma(\mu,\nu)$ is the set of couplings of $\mu$ and $\nu$, namely,
\[
\Gamma(\mu,\nu) \coloneqq \big\{\gamma\in \mathcal P (\R^2): \gamma(A\times \R) = \mu(A), 
 \; \gamma(\R \times A) = \nu(A)\text{ for all } A \in \mathcal B (\R) \big\}.
\]
We refer to \cite{AmbrosioSerfaty08} for the basic properties of $W$. As usual, we set $\Gamma_\circ (\mu,\nu) \subset \Gamma(\mu,\nu)$ as the set of transport plans $\gamma$ which minimise \eqref{W2}.

Since we are working with positive and negative particles, we follow \cite{GarroniVanMeursPeletierScardia18prep} by defining a space of probability measures on $\R \times \{\pm1\}$, where $\R \times \{\pm1\}$ is endowed with the distance 
\begin{equation*}
  \sfd^2 (\bar x, \bar y) \coloneqq |x - y |^2 + |p-q|,
  \qquad \bar x = (x, p) \in \R \times \{\pm1\}, \ \bar y = (y, q) \in \R \times \{\pm1\}.
\end{equation*}
We denote this probability space by $\mathcal{P}(\R \times \{\pm1\})$, and its elements by $\bmu$ or $(\mu^+, \mu^-)$, with the understanding that 
\begin{equation} \label{bmu:def}
  \bmu (A^+, A^-) = \mu^+(A^+) + \mu^-(A^-),  \qquad \text{for all $A^+, A^- \in \mathcal B(\R)$}.
\end{equation}
On
\begin{equation*}
  \mathcal P_2 (\R \times \{\pm1\}) \coloneqq \left\{ \bmu \in \mathcal P (\R \times \{\pm1\}) : \int_\R |x|^2 \, \de \mu^\pm (x) < +\infty \right\}
\end{equation*}
we define the (square of the) 2-Wasserstein distance between $\bmu$ and $\bnu$ as 
\begin{align} \label{bW2}
\bW^2 \big( \bmu, \bnu \big) \coloneqq \inf_{\bgamma \in \bGamma(\bmu,\bnu)} \iint_{( \R\times\{\pm1\} )^2} \sfd^2(\bar x, \bar y)\, \de \bgamma(\bar x, \bar y),
\end{align}
where $\bGamma(\bmu,\bnu)$ is the set of couplings of $\bmu$ and $\bnu$, namely,
\[
\begin{split}
\bGamma(\bmu,\bnu) \coloneqq \big\{ & \bgamma\in \mathcal P \big((\R\times\{\pm1\})^2\big): \bgamma(A\times (\R\times\{\pm1\})) = \bmu(A), \\
  &  \bgamma((\R\times\{\pm1\})\times A) = \bnu(A)\text{ for all } A \in \mathcal B (\R \times\{\pm1\}) \big\}.
\end{split}
\]

Since it turns out that \eqref{P} has a mass-preserving solution $\brho(t) := (\rho^+(t), \rho^-(t)) \in \mathcal P_2 (\R \times \{\pm1\})$, for which also the mass of $\rho^+(t)$ and $\rho^-(t)$ is conserved in time, we define the corresponding subspace
\begin{equation*}
  \mathcal P_2^m (\R \times \{\pm1\}) \coloneqq \{ \bmu \in \mathcal P_2 (\R \times \{\pm1\}) : \mu^+(\R) = m \};
\end{equation*}
where $m \in [0,1]$ is the total mass of the positive particle density. Clearly, if $\bmu \in \mathcal P_2^m (\R \times \{\pm1\})$, then $\mu^-(\R) = 1-m$. For any $\bmu, \bnu \in \mathcal P_2^m (\R \times \{\pm1\})$ we have that
\begin{equation} \label{bW:est}
  \bW^2 ( \bmu, \bnu ) \leq W^2 ( \mu^+, \nu^+ ) + W^2 ( \mu^-, \nu^- ),
\end{equation}
which simply follows by shrinking the set of couplings $\bGamma(\bmu,\bnu)$ in \eqref{bW2}.

\subsection{Weak form of the continuum problem (\ref{P})}

We use the following notation convention. For any $\brho \in \mathcal P (\R \times \{\pm1\})$, we set
\begin{equation} \label{kap:rho:trho}
  \rho  \coloneqq \rho^+ + \rho^- \in \mathcal P (\R), \qquad
  \kappa  \coloneqq \rho^+ - \rho^- \in \mathcal M (\R), \qquad
  \tilde \rho^\pm  \coloneqq [\kappa]_\pm \in \mathcal M_+ (\R).
\end{equation}
We consider the following weak form of \eqref{P}: given an initial condition $\brho^\circ \in \mathcal P_2 (\R \times \{\pm1\})$, find $\rho$ satisfying
\begin{equation} \label{wP}
\begin{split}
  0 = & \int_0^T \int_\R \partial_t \varphi^\pm (x) \, \de \rho^\pm (x) \de t \\
      & - \frac12 \int_0^T \iint_{\R \times \R} \big( (\varphi^\pm)' (x) - (\varphi^\pm)' (y) \big) \, V' (x - y) \,\de([\kappa]_\pm \otimes [\kappa]_\pm)(x,y) \de t \\
      & - \int_0^T \int_\R (\varphi^\pm)'(x) \, (W' * [\kappa]_\mp)(x) \, \de [\kappa]_\pm (x)\de t,
  \qquad \text{for all $\varphi^\pm \in \rC_c^\infty((0,T) \times \R)$,}
\end{split}
\end{equation}
where we have exploited that $V'$ is odd. 
We seek a solution of \eqref{wP} in $\AC (0, T; \mathcal P_2^m (\R \times \{\pm1\}))$ with $m = \rho^{+,\circ} (\R) \in [0,1]$.

\subsection{Several topologies and their connections}

Next we define the space of absolutely continuous curves and their metric derivatives. While the following definitions work on any complete metric space, we limit our exposition to $(\mathcal P_2 (\R \times \{\pm1\}), \bW)$. For any $1 \leq p < \infty$, $\AC^p (0, T; \mathcal P_2 (\R \times \{\pm1\}))$ denotes the space of all curves $\bmu : (0, T) \to \mathcal P_2 (\R \times \{\pm1\})$ for which there exists a function $f \in \rL^p(0,T)$ such that
\begin{equation} \label{eqn:defn:AC:bd}
  \bW \big( \bmu(s), \bmu(t) ) \leq \int_s^t |f(r)|^p \,\de r,
  \qquad \text{for all } 0 < s \leq t < T.
\end{equation}
We set $\AC (0, T; \mathcal P_2 (\R \times \{\pm1\})) \coloneqq \AC^1 (0, T; \mathcal P_2 (\R \times \{\pm1\}))$. By \cite[Theorem~1.1.2]{AmbrosioGigliSavare08}, the metric derivative
\begin{equation} \label{eqn:defn:metric:slope:chap7}
  |\bmu'|_\bW (t) \coloneqq \lim_{s \to t} \frac{ \bW \big( \bmu (s), \bmu (t) \big) }{|s-t|}
\end{equation}
is defined for any $\bmu \in \AC (0, T; \mathcal P_2 (\R \times \{\pm1\}))$ and for a.e.~$t \in (0,T)$. Moreover, $|\bmu'|_\bW$ is a possible choice for $f$ in \eqref{eqn:defn:AC:bd}.

The following theorem is a simplified version of \cite[Theorem~47.1]{Munkres00} applied to the metric space $(\mathcal P_2 (\R \times \{\pm1\}), \bW)$.
\begin{lem}[Ascoli-Arzel\`a] \label{lem:AA}
$\mathcal F \subset \rC ( [0, T]; \mathcal P_2 (\R \times \{\pm1\}) ) $ is pre-compact if and only if
\begin{enumerate}[(i)]
 \item \label{precompact} $\{ \bmu(t) : \bmu \in \mathcal F \}$ is pre-compact in $\mathcal P_2 (\R \times \{\pm1\})$ for all $t \in [0,T]$,
 \item \label{equicontinuous} $\displaystyle \forall \, \varepsilon > 0 \: \exists \, \delta > 0 \: \forall \, \bmu \in \mathcal F \: \forall \, t,s \in [0, T] : |t - s| < \delta \: \Longrightarrow \: \bW \big( \bmu(t), \bmu(s) \big) < \varepsilon$.
\end{enumerate} 
\end{lem}

The following theorem provides a lower semi-continuity result on the $L^2(0,T)$-norm of the metric derivative. We expect it to be well-known, but we only found it proven in the PhD thesis \cite[Lemma 8.2.8]{PatricksThesis}. 

\begin{thm}[Lower semi-continuity of metric derivatives] \label{thm:lsc}
Let $\bmu_n, \bmu : [0, T] \to \mathcal P_2 (\R \times \{\pm1\})$. If $\bW(\bmu_n(t), \bmu(t)) \to 0$ as $n \to \infty$ pointwise for a.e.~$t \in (0,T)$, then
\begin{equation} \label{eqn:lem:metric:der:edge}
  \liminf_{n \to \infty} \int_0^T | \bmu_n' |^2_\bW (t) \, \de t
  \geq \int_0^T | \bmu' |^2_\bW (t) \, \de t.
\end{equation}
\end{thm}

\begin{proof}
We start with several preparations. First, we take a dense subset $(t_\ell)_\ell$ of $[0, T]$ for which $\bW(\bmu_n(t_\ell), \bmu(t_\ell)) \to 0$ as $n \to \infty$ for any $\ell \in \N$. Second, without loss of generality, we assume that there exists $C > 0$ such that for all $n$
\begin{equation} \label{p:bdd:MD}
  \int_0^T | \bmu_n' |^2_\bW (t) \, \de t \leq C.
\end{equation}
In particular, this means that $\bmu_n$ has a representative in $\AC^2(0,T; \mathcal P_2 (\R \times \{\pm1\}))$ which is defined for all $t \in (0,T)$. Taking this representative, we set $D_n^\ell (t) \coloneqq {\bW} ( \bmu_n (t_\ell), \bmu_n (t) )$, and obtain from \cite[Thm.~1.1.2]{AmbrosioGigliSavare08} that
\begin{equation} \label{eqn:conj:metric:der:edge:pf:m1}
  | \bmu_n' |_{\bW} (t) = \sup_{\ell \in \N{}} \big| (D_n^\ell)' (t) \big|
  \qquad \text{for a.e.~$t \in (0,T)$.} 
\end{equation}

Next we prove \eqref{eqn:lem:metric:der:edge}. Firstly, since $\bW (\bmu_n(t), \bmu(t)) \to 0$ as $n \to \infty$ for a.e.~$t\in(0,T)$, we have for fixed $\ell\in\N{}$ and for a.e.~$t \in (0, T)$ that
\begin{equation} \label{eqn:conj:metric:der:edge:pf:0}
  \big| D_n^\ell (t) - D^\ell (t) \big| \xto{n \to \infty} 0,
  \qquad \text{where } D^\ell (t) \coloneqq {\bW} \big( \bmu (t_\ell), \bmu (t) \big). 
\end{equation}
Secondly, $\| D_n^\ell \|_{H^1(0, T)}$ and $\| D^\ell \|_{H^1(0, T)}$ are bounded uniformly in $n$ and $\ell$. To see this, we have by the definition of the metric derivative and \eqref{p:bdd:MD} that
\begin{equation*}
  D_n^\ell (t) = \bigg| \int_{t_\ell}^t | \bmu_n' |_{\bW}(s)\,\de s \bigg| \leq C \sqrt{T}.
\end{equation*}
Hence, $\| D_n^\ell \|_{L^2(0, T)}$ is uniformly bounded. 
With the characterisation of $| \bmu_n' |_{\bW}$ in \eqref{eqn:conj:metric:der:edge:pf:m1}, we estimate
\begin{equation} \label{eqn:conj:metric:der:edge:pf:1}
  C 
  \geq \int_0^T | \bmu_n' |^2_{\bW} (t) \, \de t
  \geq \int_0^T \big( (D_n^\ell)' (t) \big)^2 \, \de t
  \qquad \text{for all $\ell \in \N{}$,}
\end{equation}
and thus $\| D_n^\ell \|_{H^1(0, T)}$ is uniformly bounded. 
Therefore, in view of \eqref{eqn:conj:metric:der:edge:pf:0}, we have
\begin{equation} \label{eqn:conj:metric:der:edge:pf:1aaa}
  D_n^\ell \weakto D^\ell
  \qquad \text{in $H^1(0, T)$ as $n \to \infty$.} 
\end{equation}
In particular, we observe from \eqref{eqn:conj:metric:der:edge:pf:1aaa} that $D^\ell \in H^1(0, T)$ and that
\begin{equation*} 
  C \geq \liminf_{n \to \infty} \| D_n^\ell \|_{H^1(0, T)} \geq \| D^\ell \|_{H^1(0, T)}
  \qquad \text{for all $\ell \in \N{}$.}
\end{equation*}

To establish \eqref{eqn:lem:metric:der:edge}, we carefully perform a joint limit passage as $n \to \infty$ and a maximisation over $\ell$ in \eqref{eqn:conj:metric:der:edge:pf:1}. 
With this aim, we take a large fixed $L \in \N{}$, and choose a partition $\{A_\ell \}_{\ell=1}^{L}$ of Borel sets of $(0, T)$ such that for all $\ell = 1, \ldots, L$,
\begin{equation*}
 \big| (D^\ell)' (t) \big| 
  = \sup_{1 \leq \tilde \ell \leq L} \big| (D^{\tilde \ell})' (t) \big|
  \qquad \text{for a.e.~$t \in A_\ell$.}
\end{equation*}
We estimate
\[
  \int_0^T | \bmu_n' |^2_{\bW} (t) \, \de t
  \geq \int_0^T \sup_{1 \leq \ell \leq L} \big( (D_n^\ell)' (t) \big)^2 \, \de t 
  \geq \sum_{\ell = 1}^{L} \int_{A_\ell} \big( (D_n^\ell)' (t) \big)^2 \, \de t.
\]
Using \eqref{eqn:conj:metric:der:edge:pf:1aaa}, we pass to the limit $n \to \infty$ to obtain
\begin{equation*} 
  \liminf_{n \to \infty} \int_0^T | \bmu_n' |^2_{\bW} (t) \, \de t
  \geq \sum_{\ell = 1}^{L} \int_{A_\ell} \big( (D^\ell)' (t) \big)^2 \, \de t 
  = \int_0^T \sup_{1 \leq \ell \leq L} \big( (D^\ell)'(t) \big)^2 \, \de t.
\end{equation*}
By using the Monotone Convergence Theorem, we take the supremum over $L \in \N{}$ to deduce that
\begin{equation*} 
  \liminf_{n \to \infty} \int_0^T | \bmu_n' |^2_{\bW} (t) \, \de t
  \geq  \int_0^T \sup_{\ell \in \N{}} \big( (D^\ell)'(t) \big)^2 \, \de t.
\end{equation*}
We conclude by using \cite[Theorem~1.1.2]{AmbrosioGigliSavare08} to identify $\sup_{\ell \in \N{}} | (D^\ell)' |$ in $L^2(0,T)$ by $| \bmu' |_{\bW}$.
\end{proof}

Next we introduce the \textit{narrow convergence} of measures. For $\nu_n, \nu \in \mathcal M (\R)$, we say that $\nu_n$ converges in the narrow topology to $\nu$ (and write $\nu_n \weakto \nu$) as $n \to \infty$ if
\begin{equation*}
  \int \varphi \, \de\nu_n 
  \xto{ n \to \infty } \int \varphi \, \de\nu.
\end{equation*}
for any bounded test function $\varphi \in \rC (\R)$. The following lemma extends this notion for non-negative measures by allowing for discontinuous test functions. 

\begin{lem}[{\cite[Lemma 2.1]{Poupaud2002}}] \label{l:Schwartz81} 
Let $\nu_n \weakto \nu$ in $\mathcal M_+ (\R^d)$. Let $A \in \mathcal B(\R^d)$ such that $\nu(A) = 0$. Then for every bounded $\varphi \in \rC(\R^d \setminus A)$ it holds that
\begin{equation*}
  \int \varphi \, \de\nu_n 
  \xto{ n \to \infty } \int \varphi \, \de\nu.
\end{equation*}
\end{lem}
Proofs can be found in \cite[Theorems~62-63, chapter~IV, paragraph~6]{Schwartz1981} and in \cite{Delort1991,Gerard1992}, or \cite{Schochet1995} in the case where $A$ 
is closed. 

Finally, we state and prove a lemma which allows us to show that Assumption~\ref{ass:IC} is conserved in the limit as $n \to \infty$.
\begin{lem}[Narrow topology preserves separation of supports] \label{lem:supp:sep}
Let $(\nu_\varepsilon)_{\varepsilon > 0}, (\rho_\varepsilon)_{\varepsilon > 0} \subset \mathcal M_+ (\R)$ converge in the narrow topology as $\varepsilon \to 0$ to $\nu$ and $\rho$ respectively. If
\begin{equation*}
\forall \, \varepsilon > 0 : \sup ( \supp \nu_\varepsilon ) \leq \inf ( \supp \rho_\varepsilon ),
\end{equation*}
then also $\sup ( \supp \nu ) \leq \inf ( \supp \rho )$.
\end{lem}

\begin{proof}
We reason by contradiction. Suppose $M := \sup ( \supp \nu ) > \inf ( \supp \rho ) =: m$. Take a non-decreasing test function $\varphi \in \rC_b(\R)$ which satisfies
\begin{equation*}
   \varphi \equiv 0 \text{ on } \Big(-\infty, \frac{m + 2M}3 \Big],
   \quad \text{and} \quad
   \varphi \equiv 1 \text{ on } [M, \infty ).
\end{equation*} 
Since $M = \sup ( \supp \nu )$, it holds that $\int \varphi \, \de \nu > 0$. Hence, from $\nu_\varepsilon \xto{\varepsilon \to 0} \nu$ we infer that for all $\varepsilon$ small enough, it also holds that $\int \varphi \, \de \nu_\varepsilon > 0$, and thus \[ \sup ( \supp \nu_\varepsilon ) \geq \frac{m + 2M}3. \]
With a similar argument, we can deduce that $\inf ( \supp \rho_\varepsilon ) \leq \frac{2m + M}3$, which contradicts with $m < M$.
\end{proof}

\section{Definition and properties of the discrete problem (\ref{Pn:intro})}
\label{s:Pn}

In this section we give a rigorous definition to the discrete dynamics formally given by \eqref{Pn:intro}. 
We start by giving the definition of solution, establishing some properties of the energy $E_n$ introduced in \eqref{101}, and proving an existence and uniqueness result (see Proposition~\ref{prop:Pn}).
Finally, we state the discrete problem in the language of measures (see Proposition~\ref{prop:muPn}).

\begin{problem}\label{problemPn}
Given $(x^\circ,b^\circ)\in\R^{n}\times\{\pm1\}^n$ such that $x^\circ_1<x^\circ_2<\ldots<x^\circ_n$, find $(x,b)\colon [0,T]\to\R^n\times\{-1,0,1\}^n$ such that
\begin{equation} \label{Pn}
\begin{cases}
\displaystyle  \frac \de{\de t} x_i = - \frac1{n} \sum_{ j \,:\, b_i b_j = 1 } V' (x_i - x_j) - \frac1{n} \sum_{ j\,:\, b_i b_j = -1 } W' (x_i - x_j) \quad \text{on $(0, T) \setminus T_{\col}$} \\
  (x_i(0),b_i(0))=(x^\circ_i,b^\circ_i),
  \end{cases}
\end{equation}
for all $i = 1,\ldots,n$, where $T_{\col}$ is the jump set of $b$.
\end{problem}

We encode the annihilation rule in the solution concept below. With this aim, we set $H\colon\R\cup\{+\infty\}\to[0,1]$ as the usual Heaviside function, with $H(0) \coloneqq 0$ and $H(+\infty) \coloneqq 1$. 

\begin{defn}[Solution to Problem~\ref{problemPn}]\label{defSolPn}
We say that $(x,b)\colon [0,T]\to\R^n\times\{-1,0,1\}^n$ is a solution to Problem~\ref{problemPn} if
\begin{enumerate}[(a)]
\item $x \in \Lip([0, T];\R^{n})$; 
\item \eqref{Pn} is satisfied in the classical sense;
\item there exists a vector of collision times $\tau = (\tau_1, \ldots, \tau_n)$ with $\tau_i \in (0,T) \cup \{+\infty\}$ such that, setting 
\begin{equation}\label{108}
T_{\col} 
\coloneqq \{ \tau_i : 1 \leq i \leq n \} \setminus \{+ \infty \}
= \{t_1, t_2, \ldots, t_K\} 
\subset (0, T)
\end{equation}
with $0 
< t_1 < \ldots < t_K < 
T$, there holds
\begin{equation}\label{107}
b_i(t) \coloneqq b_i^\circ H (\tau_i - t)\qquad \text{for all $i=1,\ldots,n$};
\end{equation} 
  \item \label{prop:Pn:tk} setting $t_0\coloneqq0$, for all $k = 1, \ldots, K$,
  \[ t_k = \inf \big\{ t \in(0,T): \exists \, {(i,j)} \text{ such that } b_i(t_{k-1}) b_j(t_{k-1}) = -1 \text{ and } x_i(t) = x_j(t) \big\} > t_{k-1}; \]
  \item \label{prop:Pn:ann:pairs} at each time $t \in [0,T]$, there is a bijection 
  \[ \alpha\colon \{ i : b_i^\circ = 1, \, \tau_i \leq t \} \to\{ j : b_j^\circ = -1, \, \tau_j \leq t \} \] 
  such that 
  $x_i(t) = x_{\alpha(i)}(t)$.
\end{enumerate}
\end{defn}


\begin{rem}[Comments on Definition~\ref{defSolPn}]\label{comments:to:Pn}
\normalfont
We collect here some remarks on the notion of solution presented 
above.
\begin{itemize}
\item $\tau_i$ is the time at which particle $x_i$ gets annihilated: equation \eqref{107} describes this by putting to zero the charge $b_i$ at time $\tau_i$.
If $\tau_i=+\infty$, then it means that the particle $x_i$ does not collide in the time interval $(0,T)$.
\item $(t_k)$ is the ordered list of collision times at which at least one collision occurs.
\item In equation \eqref{Pn}, both $x_i$ and $b_i$ depend on time. However, on each open component of $(0, T) \setminus T_{\col}$, the charges $b_i$ remain constant.
\item Property \eqref{prop:Pn:tk} ensures that for each pair of two colliding particles, at least one gets annihilated. Property \eqref{prop:Pn:ann:pairs} ensures that both particles are getting annihilated, and that annihilation can only occur for colliding particles with non-zero charge.
\end{itemize}
\end{rem}

With reference to the collision times $t_1 < \ldots < t_K$ in \eqref{108}, we define the set of indices of the particles colliding at $t_k$ and its cardinality by
\begin{equation}\label{tK}
\Gamma_k\coloneqq \{i:\tau_i=t_k\}, \qquad \gamma_k\coloneqq\#\Gamma_k.
\end{equation}
We observe that $\gamma_k$ is even for every $k$ and that
\begin{equation}\label{eq:sum:gammak}
\sum_{k=1}^K \gamma_k\leq \frac{n}{2}.
\end{equation}

We first establish some properties of $E_n$ defined in \eqref{101}. For convenience, we display
\begin{equation} \label{f:En:der}
\frac\partial {\partial x_i} E_n (x; b) 
  = \frac1{n^2} \sum_{ j \,:\, b_i b_j = 1 } V' (x_i - x_j) + \frac1{n^2} \sum_{ j \,:\, b_i b_j = -1 } W' (x_i - x_j),
\end{equation}
where we rely on the choice $V'(0) = 0$.
We also introduce 
\[
  M_k : \R^n \to [0, \infty), \qquad 
  M_k(x) \coloneqq 
  \frac1n \sum_{i=1}^n |x_i|^k, \qquad k=1,2,\ldots
\]
which is the $k$-th moment of the empirical measure related to the particles $x_1,\ldots,x_n$.

\begin{lem}[Properties of $E_n$] \label{lem:En}
Let $n \geq 2$. For any $x \in \R^n$ and $b \in \{-1, 0, 1\}^n$, the following properties hold:
\begin{enumerate}[(i)]
  \item \label{lem:En:infty} $E_n(x;b) < +\infty$ if and only if $\forall \, i \neq j : x_i = x_j \: \Rightarrow \: b_i b_j \neq 1$;
  
  \item \label{lem:En:bd:below} $E_n+M_2$ is bounded from below;
  
  \item \label{lem:En:regy} $\nabla E_n$ is Lipschitz continuous on the sublevelsets of $y \mapsto E_n(y;b) + 2M_2(y)$; 
  
  \item \label{lem:En:decr:at:jump} if $E_n(x;b) < +\infty$ and if there exists an index pair $(I, J)$ which satisfies $b_I b_J = -1$ and $x_I = x_J$, then, there exists $C > 0$ independent of $n$ such that 
  \[ E_n(x; \bar b) \leq E_n(x; b) +\frac{C}{n}(M_2(x) + x_I^2 + 1), \] 
  where $\bar b$ is the modification of $b$ in which $b_I$ and $b_J$ are put to $0$.
\end{enumerate}
\end{lem}

\begin{proof}
Property \eqref{lem:En:infty} is a direct consequences of the properties of $V, W$ (see Assumption~\ref{ass:VW}). 
Property \eqref{lem:En:bd:below} is a matter of a simple estimate.
Using Assumption~\ref{ass:VW}) (in particular \eqref{402}), some manipulations inspired by \cite{Schochet1995}, and $r\mapsto r^2 - C \log r$ being bounded from below, we obtain
\begin{multline*}
E_n(x;b)+M_2(x) 
= \frac1{2n^2} \Big( \sum_{ \substack{ i \neq j \\ b_i b_j = 1 } } V (x_i - x_j) + \sum_{ \substack{ i, j \\ b_i b_j = -1 } } W(x_i - x_j) + \sum_{i,j=1}^n (x_i^2+x_j^2) \Big) \\
\geq \frac1{2n^2} \sum_{i,j=1}^n \Big( -C \big( [\log |x_i - x_j| ]_+ + 1 \big) +  \frac12 (x_i - x_j)^2 \Big)
\geq C.
\end{multline*}

Property \eqref{lem:En:regy} follows easily from property \eqref{lem:En:bd:below}.
%
To prove \eqref{lem:En:decr:at:jump}, we set $y := x_I = x_J$ and assume for convenience that $b_I = 1$ and $b_J = -1$. Then, we compute
\begin{equation*}
\begin{split} E_n(x; b) - E_n(x; \bar b) 
  ={} & \frac1{2n^2} \bigg( \sum_{ \substack{ j \neq I \\ b_j = 1 } } V (x_I - x_j) + \sum_{ \substack{ i \neq J \\ b_i = -1 } } V (x_i - x_J) \bigg) \\
  & + \frac1{2n^2} \bigg( \sum_{ j \,:\, b_j = -1 } W (x_I - x_j) + \sum_{ i \,:\, b_i = 1 } W (x_i - x_J) \bigg) - \frac{W(0)}{2 n^2} \\
  ={} & \frac1{2n^2} \bigg( \sum_{ \substack{ i = 1 \\ i \neq I, J } }^n |b_i| V (x_i - y) + \sum_{ i = 1 }^n |b_i| W (x_i - y) \bigg) - \frac{W(0)}{2 n^2} \\
  ={} & \frac1{2n^2} \sum_{\substack{i=1 \\ i\neq I,J}}^n |b_i| (V+W)(x_i-y)+ \frac{W(0)}{2 n^2} \\
  \geq{} & -\frac{C}{n^2}\sum_{i=1}^n(x_i-y)^2+ \frac{W(0)}{2 n^2} 
  \geq -\frac{C}{n} (M_2(x) + y^2 + 1),
\end{split}
\end{equation*}
where 
we have used \eqref{VWgrowth}.
\end{proof}


We now prove that Problem~\ref{problemPn} has a unique solution. In addition, we establish several properties of it. 
\begin{prop}\label{prop:Pn}
Let $n \geq 2$, $T > 0$, and $(x^\circ,b^\circ)\in\R^{n}\times\{\pm1\}^n$ be such that $x^\circ_1<x^\circ_2<\ldots<x^\circ_n$. 
Then there exists a unique 
solution $(x,b)$ to Problem~\ref{problemPn} in the sense of Definition~\ref{defSolPn}.
Moreover, the following properties are satisfied:
\begin{enumerate}[(i)] 
\item \label{prop:Pn:bound:M2} there exists $C>0$ independent of $n$ such that 
$$
  M_2(x(t)) \leq C t + M_2(x^\circ), \quad
  M_4(x(t)) \leq C t(M_2(x^\circ) + t) + M_4(x^\circ)
  \quad\text{for all }t\in[0,T];
$$ 
\item \label{prop:Pn:sepn} $\displaystyle \inf_{0 < t < T} \min \{ | x_i (t) - x_j (t)| : b_i (t) b_j (t) = 1 \} > 0$;
  \item \label{prop:Pn:En} the energy function $e \colon [0,T) \to \R$ defined by $e(t) \coloneqq E_n(x(t); b(t))$ is left-continuous on $[0, T)$, differentiable on $(0,T) \setminus T_{\col}$, and $e'(t) \leq 0$ for all $t \in (0,T) \setminus T_{\col}$. 
  Moreover, denoting by $\llbracket e(t_k)\rrbracket\coloneqq e(t_k)-e(t_k-)$ the jump of $e$ at $t_k$, we have that 
 \begin{align}\label{109}
  \llbracket e(t_k) \rrbracket 
  &\leq \frac{C}{n} \bigg(\gamma_k M_2(x(t_k)) + \gamma_k + \sum_{i\in\Gamma_k} x_i^2(t_k)\bigg) \quad\text{for every $k=1,\ldots,K$,} \\ 
  \label{111}
\sum_{k=1}^K \llbracket e(t_k)\rrbracket
&\leq C (T + M_2(x^\circ) + 1),
  \end{align}
  where $\gamma_k$ and $\Gamma_k$ are defined in \eqref{tK}, and $C > 0$ is a constant independent of $n$;
  \item \label{prop:Pn:EDI} $\displaystyle E_n (x(t);b(t)) - E_n (x^\circ;b^\circ) \leq C (t + M_2(x^\circ) + 1) - \frac1n \int_0^t |\dot x(s)|^2\,\de s$ for all $t \in (0,T]$; 
  \item \label{prop:Pn:ass} there exists an $L \in \N$ such that for all $t \in [0, T)$, $(x(t), b(t))$ satisfies Assumption~\ref{ass:IC}, \emph{i.e.}, there exist $-\infty < a_0(t) \leq a_1(t) \leq \ldots \leq a_{2L}(t) < +\infty$ such that
\begin{equation*}
  \{ x_i(t) : b_i(t) = 1 \} 
  \subset \bigcup_{\ell = 1}^L \big( a_{2 \ell - 2}(t), a_{2 \ell - 1}(t) \big), \qquad
  \{ x_i(t) : b_i(t) = -1 \} 
  \subset \bigcup_{\ell = 1}^L \big( a_{2 \ell - 1}(t), a_{2 \ell}(t) \big).
\end{equation*}
\end{enumerate}
\end{prop}

\begin{proof}
\textit{Step 1: Construction of $(x,b)$, properties \eqref{prop:Pn:bound:M2} and \eqref{prop:Pn:sepn}, and \eqref{109}}. We define the counterpart of \eqref{Pn} in which no collision occurs, \emph{i.e.}, we seek $n$ trajectories $y_i : [0, T] \to \R$ such that $y_i(0)=x_i^\circ$ and
\begin{equation} \label{p:Pn}
  \frac \de{\de t} y_i = - \frac1{n} \sum_{ j \,:\, b_i^\circ b_j^\circ = 1 } V' (y_i - y_j) - \frac1{n} \sum_{ j \,:\, b_i^\circ b_j^\circ = -1 } W' (y_i - y_j) \quad \text{on } (0,+ \infty).
\end{equation}
for all $i=1,\ldots,n$. From \eqref{f:En:der} we observe that \eqref{p:Pn} is the gradient flow of $E_n (\cdot; b^\circ)$ given by
\begin{equation}\label{p:PnGF}
\begin{cases}
\dot y(t)=-n\nabla E_n(y(t); b^\circ), \\
y(0)=x^\circ.
\end{cases}
\end{equation}
From Lemma~\ref{lem:En} we observe that \eqref{p:PnGF} has a unique, classical solution $y(t)$ locally in time. In particular, $t \mapsto E_n (y(t); b^\circ)$ 
is non-increasing.

Next we show that the solution $y$ can be extended to the complete time interval $[0,T]$. With this aim, we prove that the second moment $M_2(y(t))$ (and for later use  the fourth moment $M_4(y(t))$) are finite as long as $t\mapsto y(t)$ exists. From \eqref{p:Pn}, using \eqref{ass:VW:even} and \eqref{ass:VW:pbds}, we estimate
\begin{align*}
  \frac \de{\de t} M_2(y(t))
  &= \frac2n \sum_{i=1}^n y_i(t)\dot y_i(t) \\
  &= -\frac2{n^2} \sum_{i=1}^n \bigg( \sum_{j \,:\, b_i b_j = 1} y_i V' ( y_i - y_j ) + \sum_{j \,:\, b_i b_j = -1} y_i W' ( y_i - y_j ) \bigg) \\
  &= -\frac1{n^2} \sum_{i, j \,:\, b_i b_j = 1} (y_i - y_j) V' ( y_i - y_j ) - \frac1{n^2} \sum_{i,j \,:\, b_i b_j = -1} (y_i - y_j) W' ( y_i - y_j ) 
  \leq C,
\end{align*}
Hence, 
\begin{equation}\label{105}
M_2(y(t)) 
\leq M_2(y(0)) + Ct 
\leq M_2(x^\circ) + CT
, \qquad \text{for all $t \in [0, T]$.}
\end{equation}
Similarly, using the identity $a^3 - b^3 = (a^2 + ab + b^2)(a - b)$, we compute 
\begin{align*}
  \frac \de{\de t} M_4(y(t))
  &= \frac4n \sum_{i=1}^n y_i^3(t)\dot y_i(t) \\
  &= -\frac4{n^2} \sum_{i=1}^n \bigg( \sum_{j \,:\, b_i b_j = 1} y_i^3 V' ( y_i - y_j ) + \sum_{j \,:\, b_i b_j = -1} y_i^3 W' ( y_i - y_j ) \bigg) \\
  &= -\frac2{n^2} \sum_{i, j \,:\, b_i b_j = 1} (y_i^3 - y_j^3) V' ( y_i - y_j ) - \frac2{n^2} \sum_{i,j \,:\, b_i b_j = -1} (y_i^3 - y_j^3) W' ( y_i - y_j ) \\
 &\leq \frac C{n^2} \sum_{i, j \,:\, b_i b_j = 1} (y_i^2 +y_i y_j + y_j^2) + \frac C{n^2} \sum_{i,j \,:\, b_i b_j = -1} (y_i^2 +y_i y_j + y_j^2) \\
  &\leq \frac{C}{n^2} \sum_{i = 1}^n \sum_{j = 1}^n (y_i^2(t) + y_j^2(t))
  = C M_2(y(t)) \leq C(t + M_2(x^\circ)),
\end{align*}
where we have used \eqref{105}.
Hence, 
\begin{equation}\label{106}
M_4(y(t)) \leq  M_4(x^\circ) + CT \big (M_2(x^\circ) + T \big), \qquad \text{for all $t\in[0,T]$.} 
\end{equation}
In conclusion, \eqref{105} and \eqref{106} provide a priori bounds for $M_2(y(t))$ and $M_4(y(t))$ that are uniform in $n$ and $t$. 
Finally, from \eqref{105} and Lemma~\ref{lem:En}\eqref{lem:En:infty}--\eqref{lem:En:regy} we obtain that the solution $y$ to \eqref{p:PnGF} is defined and unique at least up to time $T$. 

Next we identify $t_1$ and choose those $b_i$ that jump at $t = t_1$ (see \eqref{107}). For this choice, it is enough to specify the collision times $\tau_i$ (see \eqref{108}). We note that
$$t^* := \inf \big\{ t \in(0,T]: \exists \, {(i,j)} : b_i^\circ b_j^\circ = -1 \text{ and } y_i(t) = y_j(t) \big\}$$
is either attained or $t^* = +\infty$. 
If $t^* \geq T$, we set $x = y$ and $\tau_i = +\infty$ for all $i$, and observe that properties \eqref{prop:Pn:tk} and \eqref{prop:Pn:ann:pairs} of Definition~\ref{defSolPn} are satisfied. 
If $t^* < T$, we observe that
$t_1$ in Definition~\ref{defSolPn}\eqref{prop:Pn:tk} has to be equal to $t^*$. 
We set $x |_{[0,t_1]} \coloneqq y|_{[0,t^*]}$ and observe from \eqref{105} and \eqref{106} that property \eqref{prop:Pn:bound:M2} is satisfied up to $t=t_1$.
For the choice of $\tau_i$, we follow the algorithm explained in Section~\ref{s:in:Pn}, \emph{i.e.}, for each pair of particles that collide at $t_1$, we set the corresponding $\tau_i$ equal to $t_1$. We choose the remaining values for $\tau_j > t_1$ later on in the construction. With this choice for $\tau_i$, it follows from the continuity of $x_i$ that properties \eqref{prop:Pn:tk} and \eqref{prop:Pn:ann:pairs} of Definition~\ref{defSolPn} are satisfied by construction. Since $E_n (x(t)) \leq E_n (x^\circ)$ 
for all $t \in [0, t_1)$, it follows 
that \eqref{prop:Pn:sepn} holds on $[0, t_1]$. 

Next we show that we can continue the construction above for $t > t_1$. First, applying
Lemma~\ref{lem:En}\eqref{lem:En:decr:at:jump} $\frac12 \gamma_1$ times (recall from \eqref{tK} that $\gamma_1$ is even), we find that 
\[
  E_n(x(t_1); b(t_1)) 
  \leq E_n(x(t_1); b(t_1-)) +\frac{C}{2n}\bigg(\gamma_1 M_2(x(t_1)) + \gamma_1 + \sum_{i\in\Gamma_1}x_i^2(t_1)\bigg).  
\]   
Hence, \eqref{109} is satisfied for $k=1$. Furthermore, we obtain that $E_n(x(t_1); b(t_1)) < \infty$, and thus we can continue the construction above for $t > t_1$ by putting $x(t_1), b(t_1)$ as the initial condition at $t = t_1$. 

Iterating over $k$, this construction identifies all $\tau_i < T$ (for $i \notin \cup_{k=1}^K \Gamma_k$, we set $\tau_i\coloneqq+\infty$) and $t_k$, and guarantees that $x$ is piecewise $\mathrm{C}^1$ on $[t_k, t_{k+1}]$ and globally Lipschitz. 
In addition, \eqref{109} holds for all $k=1,\ldots,K$.
\smallskip

\textit{Step 2: Uniqueness of $(x,b)$}. Let $x$ and $\tau$ be as constructed in Step 1, and set $b$ accordingly. Since \eqref{p:PnGF} has a unique solution,  Definition~\ref{defSolPn}\eqref{prop:Pn:tk} defines uniquely the time $t_1$ until which $x(t)$ is uniquely defined. By Definition~\ref{defSolPn}\eqref{prop:Pn:ann:pairs}, $b$ has to be constant on $[0,t_1)$. Since $x$ satisfies Property \eqref{prop:Pn:sepn} at $t = t_1$, all collisions at $t_1$ are collisions of two particles with opposite type. Then, from the explanation in Remark \ref{comments:to:Pn}, it is obvious that properties \eqref{prop:Pn:tk} and \eqref{prop:Pn:ann:pairs} of Definition~\ref{defSolPn} define uniquely the set of indices $i$ for which $\tau_i = t_1$. Hence, $b(t_1)$ is uniquely determined. We conclude by iterating over $k$.

\textit{Step 3: The remaining Properties \eqref{prop:Pn:En}--\eqref{prop:Pn:ass}}. 
Estimate \eqref{109} is already proved; 
summing over $k$ reads 
\begin{equation} \label{p:etk:summed}
\sum_{k=1}^K \llbracket e(t_k) \rrbracket 
\leq \frac{C}{n} \Big( \sum_{k=1}^K \gamma_k M_2(x(t_k)) + \sum_{k=1}^K \gamma_k + \sum_{k=1}^K \sum_{i\in\Gamma_k} x_i^2(t_k) \Big). 
\end{equation}
The first and second sums in the right-hand side above can be easily estimated using \eqref{prop:Pn:bound:M2} and \eqref{eq:sum:gammak}.
We estimate the third sum by using that the sets $\Gamma_k$ for $k=1,\ldots,K$ are disjoint, and that 
for every $k=1,\ldots,K$ and for every $i\in\Gamma_k$ we have that $x_i(t)=x_i(t_k)$ for all $t\geq t_k$. Hence, the third sum is bounded by $M_2(x(T))$. Collecting our estimates, we obtain \eqref{111} from \eqref{p:etk:summed}. 

With \eqref{prop:Pn:En} proven, we prove \eqref{prop:Pn:EDI} for $t = T$ by the following computation (the case $t < T$ follows by a similar estimate). Setting $t_{K+1}\coloneqq T$, we compute
  \begin{equation*}
  \begin{split}
   E_n (x(T);b(T)) - E_n (x^\circ;b^\circ) ={} & E_n(x(T);b(T)) - E_n(x(t_K); b(t_K)) \\
   &+ \sum_{k=1}^K \big[ \llbracket e(t_k) \rrbracket 
     + \big( E_n(x(t_k-);b(t_k-)) - E_n(x(t_{k-1});b(t_{k-1})) \big) \big] \\
     \leq{} & \sum_{k=1}^{K+1} \int_{t_{k-1}}^{t_k} \frac \de{\de t} E_n (x(t);b(t)) \, \de t + C (T + M_2(x^\circ) + 1) \\
     ={} & - \sum_{k=1}^{K+1} \frac1n \int_{t_{k-1}}^{t_k} |\dot x(t)|^2\,\de t + C (T + M_2(x^\circ) + 1) \\
     ={} & - \frac1n \int_0^T |\dot x(t)|^2\,\de t + C (T + M_2(x^\circ) + 1).
   \end{split}
   \end{equation*} 

Finally, we prove \eqref{prop:Pn:ass}. First, we claim that the strict ordering of the particles $\{x_i(t) : |b_i(t)| = 1 \}$ is conserved in time. Clearly, this ordering holds at $t = 0$. From \eqref{prop:Pn:sepn} it follows that any two particles, say with corresponding indices $i \neq j$ such that $b_i(t) b_j(t) = 1$, can never swap position. Similarly, any pair $(x_i(t), x_j(t))$ with $b_i(t) b_j(t) = -1$ cannot swap either, because \eqref{prop:Pn:tk} ensures that $b_i(t)$ and $b_j(t)$ jump to $0$ at the first $t$ at which $x_i(t) = x_j(t)$.

Next we construct $a_\ell (t)$. We start with $t = 0$, and set $a_0(0), a_1(0), \ldots$ sequentially. We set $a_0(0) \coloneqq x_1^\circ - 1$, and, if $b_1^\circ = -1$, we also put $a_1(0) \coloneqq x_1^\circ - 1$. For each pair of consecutive particles $x_i^\circ, x_{i+1}^\circ$ of opposite sign, we define a new point 
\[
a_\ell(0) \coloneqq \frac12 (x_i^\circ + x_{i+1}^\circ).
\]
If the current value of $\ell$ is odd, we define $L := (\ell + 1)/2$
and set $a_{2L}(0) \coloneqq x_n^0+1$. If  $\ell$ is even, we define $L := (\ell + 2)/2$ and set $a_{2L-1}(0) \coloneqq a_{2L}(0) \coloneqq x_n^\circ+1$. 

Since the strict ordering of the particles $\{x_i(t) : |b_i(t)| = 1 \}$ is conserved in time, we can construct $a_\ell(t)$ analogously, but for a time-dependent $L_t$. Next we show how to modify this construction such that $L_t$ can be chosen independently of $t$. Because of the ordering of $\{x_i(t) : |b_i(t)| = 1 \}$ and that its cardinality is non-increasing in time, the numbers of pairs of consecutive particles $x_i(t), x_{i+1}(t)$ of opposite non-zero charge is also non-increasing in time. Hence, $t \mapsto L_t$ is non-increasing in time. In case $L_t < L$, we modify the construction of $a_\ell (t)$ above simply by adding a surplus of points $a_\ell (t)$ which all equal $a_{2L_t} (t)$.
\end{proof}

Next we establish several properties of the empirical measures associated to the solution $(x;b)$ of Problem~\ref{problemPn} with initial condition $(x^\circ,b^\circ)$ as in Proposition~\ref{prop:Pn}. With this aim, we set $n^\pm := \{ i : b_i^\circ = \pm1 \}$ as the number of positive/negative particles at time $0$, and note that $n^+ + n^- = n$. The empirical measures associated to $(x(t); b(t))$ are
\begin{equation} \label{mupn}
  \mu_n^{\pm,\circ} \coloneqq \frac1n \sum_{i \,:\, b_i^\circ = \pm1} \delta_{x_i^\circ}, \qquad
  \mu_n^\pm(t) \coloneqq \frac1n \sum_{i \,:\, b_i^\circ = \pm1} \delta_{x_i(t)},
\end{equation}
which both have total mass equal to $n^\pm / n$ for all $t \in [0,T)$.
As in \eqref{kap:rho:trho}, we also set
\begin{equation} \label{kapn:mun:tmun}
  \kappa_n(t) \coloneqq \frac1n \sum_{i=1}^{n} b_i^\circ \delta_{x_i(t)}, \qquad
  \mu_n(t) \coloneqq \frac1n \sum_{i=1}^{n} \delta_{x_i(t)},  \qquad
  \tilde \mu_n^\pm (t) \coloneqq [\kappa_n(t)]_\pm.
\end{equation}


\begin{prop}[Proposition~\ref{prop:Pn} in terms of measures] \label{prop:muPn}
Given the setting as in Proposition~\ref{prop:Pn} with $(x, b)$ the solution to \eqref{Pn}, let $\bmu_n\coloneqq (\mu_n^+,\mu_n^-)$, $\tilde \bmu_n\coloneqq (\tilde\mu_n^+,\tilde\mu_n^-)$, and $\kappa_n$ as constructed from $(x, b)$ through \eqref{mupn} and \eqref{kapn:mun:tmun}. Then,
\begin{enumerate}[(i)]
  \item \label{prop:muPn:muka:rel} $\displaystyle\tilde \mu_n^\pm(t) = 
  \frac1n \sum_{i=1}^{n} [b_i(t)]_\pm \delta_{x_i(t)}$; 

  \item \label{prop:muPn:met:der} $\bmu_n \in \AC^2 (0,T; \mathcal P_2^m({\R^2}))$ with $m = n^+/n$, and 
  \begin{equation}\label{104}
  |\bmu_n'|_\bW^2 (t) \leq \frac1n \sum_{i=1}^n \Big( \frac \de{\de t} x_i (t) \Big)^2
  \quad \text{for all } 0<t<T;
  \end{equation}
  \item \label{prop:muPn:wP} $\bmu_n$ is a solution to $\eqref{P}$ with initial condition $\bmu_n^\circ=(\mu_n^{+,\circ},\mu_n^{-,\circ})$.
\end{enumerate}
\end{prop}

\begin{proof}
Property \eqref{prop:muPn:muka:rel} is a corollary of Proposition~\ref{prop:Pn}. 
Indeed, Proposition~\ref{prop:Pn}\eqref{prop:Pn:ass} implies that $[\kappa_n(t)]_\pm \geq \frac1n \sum_{i=1}^{n} [b_i(t)]_\pm \delta_{x_i(t)}$, while Definition~\ref{prop:Pn}\eqref{prop:Pn:ann:pairs} implies that $|\kappa_n(t)| (\R) \leq \frac1n \sum_{i=1}^{n} |b_i(t)|$. We conclude \eqref{prop:muPn:muka:rel}. 
\smallskip

Next we prove \eqref{prop:muPn:met:der}. From the definition of $\bmu_n$ in \eqref{mupn} we observe that $\bmu_n(t) \in 
\mathcal P_2^m({\R^2})$ for all $0 < t < T$. Hence, \eqref{bW:est} applies, and we obtain
\begin{equation} \label{p:bW:ito:Wpm}
  \bW^2 \big( \bmu_n (s), \bmu_n (t) \big) 
  \leq  W^2 \big( \mu_n^+ (s), \mu_n^+ (t) \big) + W^2 \big( \mu_n^- (s), \mu_n^- (t) \big)
  \quad \text{for all } 0 < s \leq t < T.
\end{equation}
To estimate the right-hand side, we let $0 < s \leq t < T$ be given, and introduce the coupling
\begin{equation*} 
  \gamma_n^\pm 
\coloneqq \frac1n \sum_{i\,:\, b_i^\circ = \pm 1} \delta_{(x_i (s), x_i (t) ) } 
\in \Gamma \big( \mu_n^\pm (s), \mu_n^\pm (t) \big).
\end{equation*}
By definition of the Wasserstein distance \eqref{W2}, we obtain
\begin{equation} \label{lem:disc:metric:der:charzn:pf:1}
  W^2 \big( \mu_n^\pm (s), \mu_n^\pm (t) \big)
  \leq \iint_{{\R^2} \times {\R^2}} |x - y|^2 \, \de\gamma_n^\pm (x, y)
  = \frac1n \sum_{i\,:\,b_i^\circ = \pm 1} 
  \big( x_i(s) - x_i(t) \big)^2.
\end{equation}
Finally, using in sequence the estimates \eqref{eqn:defn:metric:slope:chap7}, \eqref{p:bW:ito:Wpm} and \eqref{lem:disc:metric:der:charzn:pf:1}, we conclude \eqref{104}. Since $x \in \Lip ([0,T]; \R^n)$, we obtain that $\bmu_n \in \AC^2 (0,T; \mathcal P_2^m (\R^2 \times \{\pm 1\} ))$. 
\smallskip

Next we prove \eqref{prop:muPn:wP}. We rewrite \eqref{Pn} as
\begin{align*}
  \dot x_i (t) 
  &= -b_i(t) \big( V' * \tilde \mu_n^+(t) + W' * \tilde \mu_n^-(t) \big)(x_i(t)), 
  && \text{for $i$ such that  $b_i^\circ = 1$,} \\
  \dot x_i (t) 
  &= -b_i(t) \big( W' * \tilde \mu_n^+(t) + V' * \tilde \mu_n^-(t) \big)(x_i(t)),
  && \text{for $i$ such that  $b_i^\circ = -1$.}
\end{align*}
Let $\varphi \in \rC_c^\infty((0,T) \times \R)$ be any test function. Since $x_i$ is Lipschitz, the Fundamental Theorem of Calculus applies, and thus we obtain, using \eqref{prop:muPn:muka:rel}, 
\begin{align*}
  0 
  &= \frac1n \sum_{i \,:\, b_i^\circ = 1} \int_0^T \frac \de{\de t} \varphi (t, x_i(t)) \,\de t 
  = \frac1n \sum_{i \,:\, b_i^\circ = 1} \bigg[ \int_0^T \partial_t \varphi (x_i) \,\de t + \int_0^T \varphi' (x_i) \, \dot x_i \,\de t \bigg] \\
  &= \int_0^T \int_\R \partial_t \varphi \, \de \mu_n^+ \de t - \int_0^T \frac1n \sum_{i \,:\, b_i = 1} \varphi' (x_i) \, \big( V' * \tilde \mu_n^+ + W' * \tilde \mu_n^- \big)(x_i) \, \de t \\
  &= \int_0^T \int_\R \partial_t \varphi \, \de \mu_n^+ \de t - \int_0^T \int_\R \varphi' \, \big( V' * [\kappa_n]_+ + W' * [\kappa_n]_- \big) \, \de [\kappa_n]_+ \de t.
\end{align*}
Since $\varphi$ is arbitrary and $V'$ is odd, we conclude that $\mu_n^+$ satisfies \eqref{wP}. From a similar argument, it follows that also $\mu_n^-$ satisfies \eqref{wP}.
\end{proof}

\section{Statement and proof of the main convergence theorem}
\label{s:t}

In this section, we state and prove our main convergence theorem.

\begin{thm} [Discrete-to-continuum limit] \label{t} 
Let the potentials $V$ and $W$ satisfy Assumption~\ref{ass:VW}. 
Let $(x^{n,\circ}, b^{n,\circ})_n$ be a sequence of initial conditions such that 
\begin{enumerate}[(i)]
\item \label{En:bdd} $E_n(x^{n,\circ};b^{n,\circ})$ is bounded uniformly in $n$, 
\item \label{M24:bdd} $(\bmu_n^\circ)_n$ (see \eqref{mupn}) has bounded fourth moment uniformly in $n$, 
\item there exists an $L \in \N$ independent of $n$ such that Assumption~\ref{ass:IC} is satisfied for all $n$.
\end{enumerate}
Then for every $T > 0$ the curves $\bmu_n \in \AC^2(0, T; \mathcal P_2 (\R \times \{\pm1\}) )$ determined by the solution $(x^n, b^n)$ to Problem~\ref{problemPn} through \eqref{mupn} for each $n$, converge in measure uniformly in time along a subsequence to a solution $\brho$ of \eqref{wP}, whose initial condition $\brho_\circ$ is the limit of $(\bmu_n^\circ)_n$ along the same subsequence.
\end{thm}

The proof is divided in three steps. In the first step we use compactness of $\bmu_n (t)$ to extract a subsequence $n_k$ along which $\bmu_n (t)$ converges to some $\brho (t)$. In the remaining two steps we pass to the limit in \eqref{wP} as $k \to \infty$ to show that the limiting curve $\brho (t)$ also satisfies \eqref{wP}. Step 2 contains the main novelty; relying on Assumption~\ref{ass:IC} with an $n_k$-independent number $L$, we prove that $[\kappa_{n_k} (t)]_\pm \weakto [\kappa (t)]_\pm$ as $k \to \infty$ pointwise in $t$.

\begin{proof}
\textit{Step 1: $\bmu_n$ converges along a subsequence $n_k \to \infty$ in $\rC([0,T]; \mathcal P_2 (\R \times \{\pm1\}))$ to $\brho \in \AC^2 (0,T; \mathcal P_2^m (\R \times \{\pm1\}))$ with $m := \rho^{\circ, +} (\R)$.} We prove this statement by means of the Ascoli-Arzel\`a Theorem (see Lemma~\ref{lem:AA}) applied to the metric space $(\mathcal P_2 (\R \times \{\pm1\}), \bW)$.

First, we show that, for fixed $t \in [0,T]$, the sequence $(\bmu_n(t))_n$ is pre-compact in $\mathcal P_2 (\R \times \{\pm1\})$. From the assumption on the initial data and Proposition~\ref{prop:Pn}\eqref{prop:Pn:bound:M2} we observe that the second and fourth moments of the measures $\mu_n(t)$ defined in \eqref{kapn:mun:tmun}, given by
\begin{equation*}
  M_2 (x^n(t)) = \int_\R y^2 \,\de\mu_n(t)(y), \quad\quad
  M_4 (x^n(t)) = \int_\R y^4 \,\de\mu_n(t)(y),
\end{equation*}
are bounded uniformly in $n$ and $t \in [0, T]$. Then, from \cite[Lemma~B.3]{VanMeursMuntean14} and \cite[Proposition~7.1.5]{AmbrosioGigliSavare08} we find that $(\bmu_n(t))_n$ is pre-compact in the Wasserstein distance $\bW$.

Second, we show that the sequence $(\bmu_n)_n \subset \mathrm{C}([0,T]; \mathcal P_2 (\R \times \{\pm1\}))$ is equicontinuous (\emph{i.e.}, $(\bmu_n)_n$ satisfies Lemma~\ref{lem:AA}\eqref{equicontinuous}). 
For any $0 \leq s < t \leq T$, we estimate
\begin{equation}\label{400}
  \bW^2 \big( \bmu_n (t), \bmu_n (s) \big)
  \leq \bigg( \int_s^t |\bmu_n'|_\bW(r)\, \de r \bigg)^2
  \leq (t-s) \int_0^T |\bmu_n'|_\bW^2(r)\,\de r.
\end{equation}
To estimate the last integral above, we use consecutively the estimates in Proposition~\ref{prop:muPn}\eqref{prop:muPn:met:der}, 
Proposition~\ref{prop:Pn}\eqref{prop:Pn:EDI}, 
Lemma~\ref{lem:En}\eqref{lem:En:bd:below} and Proposition~\ref{prop:Pn}\eqref{prop:Pn:bound:M2} 
to obtain 
\begin{equation}\label{401}
\begin{split}
\int_0^T |\bmu_n'|_\bW^2(r)\,\de r\leq{} & \frac1n\int_0^T\sum_{i=1}^n \Big(\frac{\de}{\de t} x_i^n(r)\Big)^2\de r=\frac1n \int_0^T |\dot x^n(r)|^2\,\de r \\
\leq{} & C (T + M_2 (x^{n,\circ}) + 1) + E_n(x^{n,\circ};b^{n,\circ}) - E_n(x^n(T);b^n(T)) \\
\leq{} & C (T + M_2 (x^{n,\circ}) + 1) + E_n(x^{n,\circ};b^{n,\circ}).
\end{split}
\end{equation}
which, by the assumptions on the initial data, is bounded uniformly in $n$. Hence, the right-hand side in \eqref{400} 
is bounded by $C (t-s)$, and thus $(\bmu_n)_n$ is equicontinuous.

From the pre-compactness of $(\bmu_n(t))_n$ and the equicontinuity of $(\bmu_n)_n$, we obtain from Lemma~\ref{lem:AA} the existence of a subsequence $n_k$ along which $(\bmu_n)_n$ converges in $\mathrm{C}([0,T]; \mathcal P_2 (\R \times \{\pm1\}))$ to some limiting 
curve $\brho \in \mathrm{C}([0,T]; \mathcal P_2 (\R \times \{\pm1\}))$. 
In fact, 
combining the lower semi-continuity obtained in Theorem~\ref{thm:lsc} with \eqref{401}, we obtain that 
$\brho \in \mathrm{AC}^2 (0,T; \mathcal P_2 (\R \times \{\pm1\}))$. Moreover, since the total mass of $\mu_n^+(t)$ is conserved in time, and since the narrow topology conserves mass, we conclude that $\brho(t) \in \mathcal{P}_2^m (\R \times \{\pm1\})$ for all $t \in [0,T]$. This completes the proof of Step 1. For later use, we set as in \eqref{kap:rho:trho}
\begin{equation*}
  \rho \coloneqq \rho^+ + \rho^-, \qquad
  \kappa \coloneqq \rho^+ - \rho^-, \qquad
  \tilde \rho^\pm \coloneqq [\kappa]_\pm.
\end{equation*}

\smallskip

\textit{Step 2: $\tilde \bmu_{n_k} (t) \weakto \tilde \brho (t)$ as $k \to \infty$ pointwise for all $t \in [0, T]$.} We set $\tilde \mu_{n_k}^\pm = [\kappa_{n_k}]_\pm$ as in \eqref{kapn:mun:tmun}.  
We keep $t \in [0, T]$ fixed, and remove it from the notation in the remainder of this step. The structure of the proof of Step 2 is to show by compactness that $(\tilde \bmu_{n_k})_k$ has a converging subsequence, and to characterise the limit as $\tilde \brho$. Since $\tilde \brho$ is independent of the choice of subsequence, we then conclude that the full sequence $(\tilde \bmu_{n_k})_k$ converges to $\tilde \brho$. Keeping this in mind, in the following we omit all labels of subsequences of $n$.

Since the second moments of $\tilde \bmu_n$ are obviously bounded by $M_2 (x^n)$, the sequence $(\tilde \bmu_n)$ is tight, and thus, by Prokhorov's Theorem, $(\tilde \bmu_n)$ converges narrowly along a subsequence to some $\tilde \bmu \in \mathcal M_+ (\R \times \{\pm1\})$.

We claim that $\tilde \bmu$ does not have atoms.
We reason by contradiction. 
Suppose that $\tilde \mu^+$ has an atom at $y$ of mass $\alpha > 0$ (the case of $\tilde \mu^-$ can be treated analogously). 
Then, setting $B_\eta(y)$ as the ball around $y$ with radius $\eta$, we infer from $\tilde \mu_n^+ \weakto \tilde \mu^+$ that $\liminf_{n \to \infty} \tilde \mu_n^+ (B_\eta(y)) \geq \alpha > 0$ for any $\eta > 0$. 
By choosing $\eta > 0$ small enough, the contribution of the particles in $B_\eta(y)$ to the energy $E_n(x^n; b^n)$ can be made arbitrarily large, which contradicts with the uniform bound on $E_n(x^n; b^n)$ given by Proposition~\ref{prop:Pn}\eqref{prop:Pn:EDI}.

In the remainder of this step we show that $\tilde \mu^\pm = [\kappa]_\pm$, regardless of the choice of the subsequence. It is enough to show that
\begin{align} \label{p:S2:1}
  [\kappa]_\pm &\leq \tilde \mu^\pm \\\label{p:S2:2}
  [\kappa]_\pm (\R) &\geq \tilde \mu^\pm (\R)
\end{align}
Regarding \eqref{p:S2:1}, we obtain
from Step 1 that
\begin{equation*}
  \tilde \mu_n^+ - \tilde \mu_n^- 
  = \kappa_n 
  \weakto \kappa \quad \text{as } n \to \infty.
\end{equation*}
Hence, $\tilde \mu^+ - \tilde \mu^- = \kappa$, which implies \eqref{p:S2:1}. 
To prove \eqref{p:S2:2}, we let $\{ a_\ell^n \}_{\ell = 0}^{2L}$ be as in Proposition~\ref{prop:Pn}\eqref{prop:Pn:ass}, and set
\begin{equation*}
  \tilde \mu_n^\ell := \left\{ \begin{array}{ll}
    \tilde \mu_n^+ |_{(a_{\ell - 1}^n, a_\ell^n)}
    & \ell \text{ odd} \\
    \tilde \mu_n^- |_{(a_{\ell - 1}^n, a_\ell^n)}
    & \ell \text{ even}
  \end{array} \right.
\end{equation*}
for all $\ell \in \{1, \ldots, 2L\}$. By construction, 
\[ \sum_{\ell = 1}^L \tilde \mu_n^{2\ell - 1} = \tilde \mu_n^+
  \quad \text{and} \quad
  \sum_{\ell = 1}^L \tilde \mu_n^{2\ell} = \tilde \mu_n^-. \]
Together with $\tilde \bmu_n \weakto \tilde \bmu$, we conclude that $(\tilde \mu_n^\ell)_n$ are tight for any $\ell$, and thus, applying Prokhorov's Theorem once more, each sequence $(\tilde \mu_n^\ell)_n$ converges along a subsequence in the narrow topology to some $\tilde \mu^\ell \in \mathcal M_+ (\R)$. In particular, from $\tilde \bmu_n \weakto \tilde \bmu$ and 
\[ \tilde \mu_n^- = \sum_{\ell = 1}^L \tilde \mu_n^{2\ell} \weakto \sum_{\ell = 1}^L \tilde \mu^{2\ell}, \]
we infer that $\tilde \mu^- = \sum_{\ell = 1}^L \tilde \mu^{2\ell}$. By a similar argument, it  follows that $\tilde \mu^+ = \sum_{\ell = 1}^L \tilde \mu^{2\ell - 1}$. Finally, since $\sup (\supp \tilde \mu_n^\ell ) < \inf (\supp \tilde \mu_n^{\ell+1} )$ for all $1 \leq \ell \leq 2L - 1$, we obtain from Lemma~\ref{lem:supp:sep} that $\sup (\supp \tilde \mu^\ell ) < \inf (\supp \tilde \mu^{\ell+1} )$ for all $1 \leq \ell \leq 2L - 1$. Hence, there exists $A := \{ a_\ell \}_{\ell = 1}^{2L-1}$ such that
\begin{multline*} 
  \supp \tilde \mu^+ \cap \supp \tilde \mu^- 
  = \bigg( \bigcup_{\ell = 1}^L \supp \tilde \mu^{2 \ell - 1} \bigg) \cap \bigg( \bigcup_{k = 1}^L \supp \tilde \mu^{2 k} \bigg) \\
  = \bigcup_{\ell = 1}^L \bigcup_{k = 1}^L \big( \supp \tilde \mu^{2 \ell - 1} \cap \supp \tilde \mu^{2 k} \big)
  = \bigcup_{\ell = 1}^{2L-1} \big( \supp \tilde \mu^\ell \cap \supp \tilde \mu^{\ell + 1} \big)
  \subset A.
\end{multline*}
Since $\tilde \mu^\pm$ does not have atoms, $\tilde \mu^\pm(A) = 0$. Together with $\tilde \mu^+ - \tilde \mu^- = \kappa$, it is easy to construct a Hahn decomposition of $\kappa$ (see, \emph{e.g.}, \cite[Theorem~6.14]{Rudin87}). 
We conclude \eqref{p:S2:2}.
\smallskip

\textit{Step 3: $\brho$ is a solution to \eqref{P}.} To ease notation, we replace $n_k$ by $n$. We show that $\brho$ satisfies \eqref{wP}. With this aim, let $\varphi^\pm \in \mathrm{C}_c^\infty((0,T) \times \R)$ be arbitrary. 
We recall from Proposition~\ref{prop:muPn}\eqref{prop:muPn:wP} that $\bmu_n$ satisfies
\begin{equation} \label{p:wP}
\begin{split}
0 = & \int_0^T \int_\R \partial_t \varphi^\pm (x) \, \de \mu_n^\pm(x) \de t - \int_0^T \int_\R (\varphi^\pm)'(x) \, (W' * [\kappa_n]_\mp)(x) \, \de [\kappa_n]_\pm (x) \de t \\
&- \frac12 \int_0^T \iint_{\R \times \R} \big( (\varphi^\pm)' (x) - (\varphi^\pm)' (y) \big) \, V' (x - y) \, \de([\kappa_n]_\pm \otimes [\kappa_n]_\pm)(x,y) \de t.
\end{split} 
\end{equation}

We show that we can pass to the limit in all three terms separately. From Step 1 it follows that $\bmu_n \weakto \brho$, and thus the limit of the first integral equals
\begin{equation*}
  \int_0^T \int_\R \partial_t \varphi^\pm (x) \, \de \rho^\pm(x) \de t.
\end{equation*}
Regarding the other two integrals in \eqref{p:wP}, we recall from Step 2 that $[ \kappa_n (t) ]_\pm \weakto [ \kappa (t) ]_\pm$ as $n \to \infty$ pointwise for all $t \in [0, T]$. Then, for the second term, since $(x,y) \mapsto (\varphi^\pm)'(x) \, W'(x-y)$ is bounded and continuous on $\R^2$, we obtain that
\begin{multline*}
  \int_\R (\varphi^\pm)'(x) \, (W' * [\kappa_n]_\mp)(x) \, \de [\kappa_n]_\pm (x)
  = \iint_{\R^2} (\varphi^\pm)'(x) \, W'(x-y) \, \de([\kappa_n]_\pm \otimes [\kappa_n]_\mp)(x,y) \\
  \xto{n \to \infty} \iint_{\R^2} (\varphi^\pm)'(x) \, W'(x-y) \, \de([\kappa]_\pm \otimes [\kappa]_\mp)(x,y)
  = \int_\R (\varphi^\pm)'(x) \, (W' * [\kappa]_\mp)(x) \, \de [\kappa]_\pm (x).
\end{multline*}
Finally, we pass to the limit in the third integral in \eqref{p:wP}. 
We employ Lemma~\ref{l:Schwartz81} with $d = 2$ and $\Delta = \{ (y,y) : y \in \R \}$ the diagonal in $\R^2$. 
To show that the conditions of Lemma~\ref{l:Schwartz81} are satisfied, we observe from the fact that $r \mapsto r V'(r)$ is bounded and 
{belongs to} $\mathrm{C}(\R \setminus \{0\})$, it holds that $(x,y) \mapsto [ (\varphi^\pm)' (x) - (\varphi^\pm)' (y) ] \, V' (x - y)$  is bounded and 
{belongs to} $\mathrm{C}(\R^2 \setminus \Delta)$. 
Moreover, by Step 2, $( [ \kappa ]_\pm \otimes [ \kappa ]_\pm ) (\Delta) = (\tilde \mu^\pm \otimes \tilde \mu^\pm)(\Delta)=0$. 
Hence, by Lemma~\ref{l:Schwartz81} we can pass to the limit in the third term in \eqref{p:wP}, whose limit reads
\begin{equation*}
  - \frac12 \int_0^T \iint_{\R \times \R} \big( (\varphi^\pm)' (x) - (\varphi^\pm)' (y) \big) \, V' (x - y) \, \de ([\kappa]_\pm \otimes [\kappa]_\pm)(x,y) \de t.
\end{equation*}
Combining the three limits above, and recalling the time regularity of $\brho$ from Step 1, we conclude that $\brho$ is a solution to \eqref{P}.
\end{proof}


\medskip

\noindent\textbf{Acknowledgments}.
The authors wish to thank the Department of Mathematics at Kanazawa University and the Zentrum Mathematik at Technische Universit\"at M\"unchen, where this research was developed.
PvM is supported by the International Research Fellowship of the Japanese Society for the Promotion of Science, together with the JSPS KAKENHI grant 15F15019.
MM is a member of the Gruppo Nazionale per l'Analisi Matematica, la Probabilit\`a e le loro Applicazioni (GNAMPA) of the Istituto Nazionale di Alta Matematica (INdAM).
Moreover, he acknowledges partial support from the ERC Starting grant \emph{High-Dimensional Sparse Optimal Control} (Grant agreement no.: 306274) and the DFG Project \emph{Identifikation von Energien durch Beobachtung der zeitlichen Entwicklung von Systemen} (FO 767/7).


\bibliographystyle{alpha} 
\bibliography{refsPatrick}

\begin{thebibliography}{GvMPS18}

\bibitem[AGS08]{AmbrosioGigliSavare08}
L.~Ambrosio, N.~Gigli, and G.~Savar{\'e}.
\newblock {\em Gradient Flows: In Metric Spaces and in the Space of Probability
  Measures}.
\newblock Birkh\"auser Verlag, New York, 2008.

\bibitem[AMS11]{AmbrosioMaininiSerfaty11}
L.~Ambrosio, E.~Mainini, and S.~Serfaty.
\newblock Gradient flow of the {C}hapman--{R}ubinstein--{S}chatzman model for
  signed vortices.
\newblock {\em Ann. Inst. H. Poincar\'{e} Anal. Non Lin\'{e}aire},
  28(2):217--246, 2011.

\bibitem[AS08]{AmbrosioSerfaty08}
L.~Ambrosio and S.~Serfaty.
\newblock A gradient flow approach to an evolution problem arising in
  superconductivity.
\newblock {\em Comm. Pure Appl. Math.}, 61(11):1495--1539, 2008.

\bibitem[BKM10]{BilerKarchMonneau10}
P.~Biler, G.~Karch, and R.~Monneau.
\newblock Nonlinear diffusion of dislocation density and self-similar
  solutions.
\newblock {\em Comm. Math. Phys.}, 294(1):145--168, 2010.

\bibitem[CXZ16]{ChapmanXiangZhu15}
S.~J. Chapman, Y.~Xiang, and Y.~Zhu.
\newblock Homogenization of a row of dislocation dipoles from discrete
  dislocation dynamics.
\newblock {\em SIAM J. Appl. Math.}, 76(2):750--775, 2016.

\bibitem[Del91]{Delort1991}
J.-M. Delort.
\newblock Existence de nappes de tourbillon en dimension deux.
\newblock {\em J. Amer. Math. Soc.}, 4(3):553--586, 1991.

\bibitem[Due16]{Duerinckx16}
M.~Duerinckx.
\newblock Mean-field limits for some {R}iesz interaction gradient flows.
\newblock {\em SIAM J. Math. Anal.}, 48(3):2269--2300, 2016.

\bibitem[FIM09]{ForcadelImbertMonneau09}
N.~Forcadel, C.~Imbert, and R.~Monneau.
\newblock Homogenization of some particle systems with two-body interactions
  and of the dislocation dynamics.
\newblock {\em Discrete Contin. Dyn. Syst.}, 23(3):785--826, 2009.

\bibitem[G{\'{e}}r92]{Gerard1992}
P.~G{\'{e}}rard.
\newblock R\'esultats r\'ecents sur les fluides parfaits incompressibles
  bidimensionnels (d'apr\`es {J}.-{Y}.\ {C}hemin et {J}.-{M}.\ {D}elort).
\newblock {\em Ast\'erisque}, 206:Exp.\ No.\ 757, 5, 411--444, 1992.
\newblock S\'eminaire Bourbaki, Vol. 1991/92.

\bibitem[GLP10]{GarroniLeoniPonsiglione10}
A.~Garroni, G.~Leoni, and M.~Ponsiglione.
\newblock Gradient theory for plasticity via homogenization of discrete
  dislocations.
\newblock {\em J. Eur. Math. Soc. (JEMS)}, 12(5):1231--1266, 2010.

\bibitem[GPPS13]{GeersPeerlingsPeletierScardia13}
M.~G.~D. Geers, R.~H.~J. Peerlings, M.~A. Peletier, and L.~Scardia.
\newblock Asymptotic behaviour of a pile-up of infinite walls of edge
  dislocations.
\newblock {\em Arch. Ration. Mech. Anal.}, 209:495--539, 2013.

\bibitem[GvMPS18]{GarroniVanMeursPeletierScardia18prep}
A.~Garroni, P.~{v}an Meurs, M.~A. Peletier, and L.~Scardia.
\newblock Convergence and non-convergence of many-particle evolutions with
  multiple signs.
\newblock {\em In preparation}, 2018.

\bibitem[Hal11]{Hall11}
C.~L. Hall.
\newblock Asymptotic analysis of a pile-up of regular edge dislocation walls.
\newblock {\em Mater. Sci. Eng., A}, 530:144--148, 2011.

\bibitem[Hau09]{Hauray09}
M.~Hauray.
\newblock Wasserstein distances for vortices approximation of {E}uler-type
  equations.
\newblock {\em Math. Models Methods Appl. Sci.}, 19(08):1357--1384, 2009.

\bibitem[HCO10]{HallChapmanOckendon10}
C.~L. Hall, S.~J. Chapman, and J.~R. Ockendon.
\newblock Asymptotic analysis of a system of algebraic equations arising in
  dislocation theory.
\newblock {\em SIAM J. Appl. Math.}, 70(7):2729--2749, 2010.

\bibitem[Hea72]{Head72III}
A.~K. Head.
\newblock Dislocation group dynamics {III}. {S}imilarity solutions of the
  continuum approximation.
\newblock {\em Philosophical Magazine}, 26(1):65--72, 1972.

\bibitem[HL82]{HirthLothe82}
J.~P. Hirth and J.~Lothe.
\newblock {\em Theory of Dislocations}.
\newblock John Wiley \& Sons, New York, 1982.

\bibitem[LMSZ18]{LMSZ18}
I.~Lucardesi, M.~Morandotti, R.~Scala, and D.~Zucco.
\newblock Upscaling of screw dislocations with prescribed external strain.
\newblock {\em Submitted}, 2018.

\bibitem[MP12a]{MonneauPatrizi12a}
R.~Monneau and S.~Patrizi.
\newblock Derivation of {O}rowan's law from the {P}eierls--{N}abarro model.
\newblock {\em Comm. Partial Differential Equations}, 37(10):1887--1911, 2012.

\bibitem[MP12b]{MonneauPatrizi12}
R.~Monneau and S.~Patrizi.
\newblock Homogenization of the {P}eierls--{N}abarro model for dislocation
  dynamics.
\newblock {\em J. Differential Equations}, 253(7):2064--2105, 2012.

\bibitem[MPS17]{MoraPeletierScardia17}
M.~G. Mora, M.~A. Peletier, and L.~Scardia.
\newblock Convergence of interaction-driven evolutions of dislocations with
  {W}asserstein dissipation and slip-plane confinement.
\newblock {\em SIAM J. Math. Anal.}, 49(5):4149--4205, 2017.

\bibitem[Mun00]{Munkres00}
J.~R. Munkres.
\newblock {\em Topology}.
\newblock Prentice Hall, Inc., Upper Saddle River, 2000.

\bibitem[Nab47]{Nabarro47}
F.~R.~N. Nabarro.
\newblock Dislocations in a simple cubic lattice.
\newblock {\em Proc. Phys. Soc.}, 59(2):256, 1947.

\bibitem[Pei40]{Peierls40}
R.~Peierls.
\newblock The size of a dislocation.
\newblock {\em Proc. Phys. Soc.}, 52(1):34--37, 1940.

\bibitem[Pou02]{Poupaud2002}
F.~Poupaud.
\newblock Diagonal defect measures, adhesion dynamics and {E}uler equation.
\newblock {\em Methods Appl. Anal.}, 9(4):533--561, 2002.

\bibitem[Rud87]{Rudin87}
W.~Rudin.
\newblock {\em Real and Complex Analysis}.
\newblock McGraw-Hill International Editions, Mathematics Series. McGraw-Hill,
  Inc., Singapore, third edition, 1987.

\bibitem[SBO07]{SmetsBethuelOrlandi07}
D.~Smets, F.~Bethuel, and G.~Orlandi.
\newblock Quantization and motion law for {G}inzburg--{L}andau vortices.
\newblock {\em Arch. Ration. Mech. Anal.}, 183(2):315--370, 2007.

\bibitem[Sch81]{Schwartz1981}
L.~Schwartz.
\newblock {\em Cours d'analyse. 1}.
\newblock Hermann, Paris, second edition, 1981.

\bibitem[Sch95]{Schochet1995}
S.~Schochet.
\newblock The weak vorticity formulation of the {$2$}-{D} {E}uler equations and
  concentration-cancellation.
\newblock {\em Comm. Partial Differential Equations}, 20(5-6):1077--1104, 1995.

\bibitem[Sch96]{Schochet96}
S.~Schochet.
\newblock The point-vortex method for periodic weak solutions of the 2-d euler
  equations.
\newblock {\em Comm. Pure Appl. Math.}, 49(9):911--965, 1996.

\bibitem[Ser07]{Serfaty07II}
S.~Serfaty.
\newblock Vortex collisions and energy-dissipation rates in the
  {G}inzburg--{L}andau heat flow. {P}art {II}: {T}he dynamics.
\newblock {\em J. Eur. Math. Soc. (JEMS)}, 9(3):383--426, 2007.

\bibitem[SS15]{SandierSerfaty152D}
E.~Sandier and S.~Serfaty.
\newblock 2{D} {C}oulomb gases and the renormalized energy.
\newblock {\em Ann. Probab.}, 43(4):2026--2083, 2015.

\bibitem[vM15]{PatricksThesis}
P.~{v}an Meurs.
\newblock {\em Discrete-to-Continuum Limits of Interacting Dislocations}.
\newblock PhD thesis, TU Eindhoven, 2015.

\bibitem[vM18]{vanMeurs18}
P.~van Meurs.
\newblock Many-particle limits and non-convergence of dislocation wall
  pile-ups.
\newblock {\em Nonlinearity}, 31:165--225, 2018.

\bibitem[vMM14]{VanMeursMuntean14}
P.~{v}an Meurs and A.~Muntean.
\newblock Upscaling of the dynamics of dislocation walls.
\newblock {\em Adv. Math. Sci. Appl.}, 24(2):401--414, 2014.

\bibitem[vMMP14]{VanMeursMunteanPeletier14}
P.~{v}an Meurs, A.~Muntean, and M.~A. Peletier.
\newblock Upscaling of dislocation walls in finite domains.
\newblock {\em European J. Appl. Math.}, 25(6):749--781, 2014.

\end{thebibliography}

\end{document}